\documentclass[letterpaper,draft,10pt]{amsart}
\usepackage{amssymb}
\usepackage{amsthm}
\usepackage{amsrefs}
\usepackage{dsfont}
\usepackage[all]{xy}
\usepackage{mathrsfs}
\usepackage{stmaryrd}

\theoremstyle{plain}
 \newtheorem{thm}{Theorem}[section]
 \newtheorem{prop}[thm]{Proposition}
 \newtheorem{lem}[thm]{Lemma}
 \newtheorem{cor}[thm]{Corollary}

 \newtheorem{dfn}[thm]{Definition}
 \newtheorem{rem}[thm]{Remark} 
 \numberwithin{equation}{section}

\renewcommand{\leq}{\leqslant}
\renewcommand{\geq}{\geqslant}

\newcommand{\THH}{\mathrm{THH}}
\newcommand{\HH}{\mathrm{HH}}
\newcommand{\im}{\mathrm{im}}
\newcommand{\Gal}{\mathrm{Gal}}
\newcommand{\Tor}{\mathrm{Tor}}
\newcommand{\Ann}{\mathrm{Ann}}
\newcommand{\Tr}{\mathrm{Tr}}
\newcommand{\Trd}{\mathrm{Trd}}
\newcommand{\Trace}{\mathrm{Trace}}
\newcommand{\op}{\mathrm{op}}
\renewcommand{\P}{\mathfrak{P}}
\newcommand{ \Ap }{A/(p)}
\newcommand{\Z}{\mathbb{Z}}
\newcommand{\Q}{\mathbb{Q}}
\newcommand{\F}{\mathbb{F}}

\newcommand{\ZAp}{\mathbb{Z}(A/(p))}
\newcommand{\blb}{\Big(}
\newcommand{\brb}{\Big)}
\newcommand{\cancel}{}
\title[THH of Maximal Orders in Simple $Q$-algebras]{Topological Hochschild Homology of Maximal Orders in Simple $\mathbb{Q}$-algebras.}

\author[H. Chan and A. Lindenstrauss]{\bfseries Henry Yi-Wei Chan and Ayelet Lindenstrauss}

\begin{document}

\vspace{18mm}
\setcounter{page}{1}
\thispagestyle{empty}

\begin{abstract} 
We calculate the topological Hochschild homology groups of a maximal order in a simple algebra over the rationals.   Since the positive-dimensional  $\THH$ groups consist only of torsion, we do this one prime ideal at a time for all the nonzero prime ideals in the center of the maximal order.   This allows us to reduce the problem to studying the topological Hochschild homology groups of maximal orders $ A $ in simple $\mathbb{Q}_p$-algebras.  We show that  the topological Hochschild homology of $ A /(p)$ splits as the tensor product of its Hochschild homology with $\THH_*(\mathbb{F}_p)$.  We use this result in Brun's spectral sequence to calculate
$\THH_*(  A ;  A /(p))$, and then we analyze the torsion to get $\pi_*( \THH(  A )^\wedge_p)$.
\end{abstract}

\maketitle

\section{Introduction}

Throughout this paper let $B$ be a central simple $\mathbb{Q}$-algebra, and let  $U$ be a maximal order in $B$.  We calculate the homotopy groups of the topological Hochschild homology spectrum $\THH(U)$. 

Topological Hochschild homology is the ring spectrum analog of Hochschild homology for rings.  It was originally defined by B\"okstedt in \cite{Bo}, and after the introduction of a strictly  associative product on ring spectra by Elmendorf, Kriz, Mandell, and May in \cite{EKMM}, it can be defined completely analogously to the definition of the Hochschild homology of a ring.  When we talk about topological Hochschild homology of a ring, we actually mean the topological Hochschild homology of that ring's Eilenberg Mac Lane spectrum.  Its homotopy groups turn out to be a finer and more interesting invariant than the Hochschild homology groups of the ring.  Moreover, as conjectured by  Tom Goodwillie, the Dennis trace map from algebraic K-theory to Hochschild homology factors through topological Hochschild homology.  So topological Hochschild homology is a closer approximation of algebraic K-theory which while being harder to calculate than Hochschild homology  is still much easier to calculate than algebraic K-theory.  As will be discussed later in this introduction, the map from the topological Hochschild homology of the maximal order $U$ to its Hochschild homology vanishes in high enough dimension, showing that the original Dennis trace map is also trivial in those high dimensions.   Work of B\"okstedt, Hsiang, and Madsen in \cite{BHM93} further refines the Dennis trace by factoring it through topological cyclic homology, which is an excellent approximation of algebraic K-theory.

 Our main result is:

\noindent {\bf{Theorem
\ref{main}}.}
Let $B$ be a simple algebra over $\mathbb{Q}$, and let $U$ be a maximal order in it.  Let $C$ be the center of $B$, and let $V$ be its ring of integers.   For every nontrivial prime ideal $\P\subset V$, the completion $B^\wedge_\P$ is a central simple $C^\wedge_\P$-algebra, and so $B^\wedge_\P$ is isomorphic to a matrix ring on some central division algebra $D_\P$ over $C^\wedge_\P$, of degree $e_\P$.  Let $\F_\P= V/\P$.  Then we have $V$-module isomorphisms
$$
\THH_*(U)\cong
\begin{cases}
V \oplus  \bigoplus_{\P\subset V \ \mathrm{prime}} \F_\P ^{\oplus e_\P-1} & *=0\\ 
 \THH_{2a-1} (V)  & *=2a-1> 0\\
  \bigoplus_{\P\subset V \ \mathrm{prime}} \F_\P ^{\oplus e_\P-1} & *=2a>0\\
 0 & *<0.
 \end{cases}
$$

\smallskip\noindent For the number ring $V$, its topological Hochschild homology was calculated by Ib Madsen and the second author,

\noindent{\bf{Theorem 1.1 of \cite{LM00}.}}
The nonzero topological Hochschild homology groups of a number ring $V$ are 
$$\THH_0(V)=V, \quad \THH_{2a-1} (V) = {\mathscr{D}}_V^{-1} /aV \quad(a>0),$$
where ${\mathscr{D}}_V$ is the different ideal.

\smallskip
There are many descriptions of the different ideal, but in the case that $V$ is of the form $\Z[x]/f(x)$ for some monic polynomial $f$ with integer coefficients, ${\mathscr{D}}_V$ is the ideal in $V$ generated by the derivative $f'(x)$ and particularly if $V=\Z$, ${\mathscr{D}}_V=\Z$ as well.  
 
Note that being a maximal order in $B$ does not uniquely determine $U$, even up to isomorphism.  The homotopy groups we get in our calculation of $\THH_*(U)$ will, however, be isomorphic for all the maximal orders $U$ in a fixed $B$---see Remark \ref{indep} below.   

The maximal order $U$ is said to be \emph{ramified} at a prime ideal $\P$ of its center $V$ if the degree $e_\P$ of the division algebra $D_\P$  over its center is greater than one.  Note that if $U$ is unramified at $\P$, $B^\wedge_P$ is isomorphic to the  ring of $i\times i$ matrices over $C^\wedge_P$ for some positive integer $i$, and then the same must be true for their valuation rings $U^\wedge_P$ and $V^\wedge_P$ by \cite[Theorem X.1]{Weil}.  Then   by Morita equivalence \cite[Proposition 3.9]{BHM93}, 
$\THH(U^\wedge_P)\simeq \THH(M_i(V ^\wedge_P))\simeq \THH(V^\wedge_P)$, which is reflected in the  fact that there is no $\P$-torsion in even dimensions in the result of Theorem \ref{main}.

 In even dimensions, we get that  $\THH_{2a}(U) \cong \HH_{2a}(U)$, which was calculated by Michael Larsen in  \cite{L95} and is given in Equation (\ref{globalHH}) below.  However, the
 linearization map $\THH_*(U)\to \HH_*(U)$ induced by sending  $(HU)^{\wedge (\ell+1)}$ to its components $U^{\otimes (\ell+1)}$ in each simplicial degree $\ell$ does not induce this isomorphism---in fact, it becomes trivial on the even-dimensional $p$-torsion when $*>  2p-1$.  It therefore becomes the zero map in even dimensions if $*> 2p-1$ when $p$ is the greatest  prime for which $e_\P >1$ for some prime ideal $\P$ of $V$ which contains $(p)$.  In odd dimensions, we just get the linearization map for the center, $\THH_{2a-1}(V) \to \HH_{2a-1}(V)$, which also becomes the zero map for high enough dimensions.
Since the Dennis trace map from algebraic K-theory to Hochschild homology factors through topological Hochschild homology via the linearization map, this gives topological Hochschild homology the potential of being a much better approximation to the higher algebraic K-groups than Hochschild homology is.

We begin with  the spectral sequence from \cite[Corollary 3.3]{Li00},
\begin{equation}\label{SpecSeq}
E_{r,s}^2 = \HH_r(U; \THH_s(\mathbb{Z} ; U))\Rightarrow \THH_{r+s}(U).
\end{equation}

It shows that $\THH_0(U)=\HH_0(U;  \THH_0(\mathbb{Z} ; U))\cong \HH_0(U)\cong U/[U,U]$.  By   \cite{L95}, $\HH_0(U)$ consists of $V$ and of torsion, and $\HH_r(U)$ consists of torsion for $r>0$.  By Marcel  B\"{o}kstedt's calculation  in \cite{Bo},  $\THH_s(\mathbb{Z})$ is torsion for $s>0$. 
So the spectral sequence shows that $\THH_0(U)$ consists of $V\oplus {\mathrm {torsion}}$ and that  for $*>0$, $\THH_*(U)$ consists entirely of torsion.   
To understand the $p$-torsion for a prime $p$, by  \cite[Addendum 6.2]{HM97}, we know that 
\begin{equation}\label{localize}
\THH(U)^{\wedge}_p\simeq \THH(U\otimes \mathbb{Z}_p)^{\wedge}_p,
\end{equation}
so it is enough to study the $p$-torsion in $\THH_*(U\otimes \mathbb{Z}_p)$.

The ideal $(p)\subseteq V$ breaks down as $\P_1^{a_1}\cdots \P_k^{a_k}$ for distinct prime ideals $\P_i$ in $V$.  Then 
$V\otimes\Z_p\cong V^\wedge_{(p)} \cong \bigoplus_{i=1}^k V^\wedge_{\P_i}$ 
and 
$U\otimes\Z_p\cong U^\wedge_{(p)} \cong \bigoplus_{i=1}^k U^\wedge_{\P_i}$ 
and so
$$\THH_*(U\otimes\Z_p) \cong  \bigoplus_{i=1}^k \THH_*(U^\wedge_{\P_i})$$
as $V$-modules, and we need to compute the latter.  Since $U$ is a central simple $V$-algebra, for every $i$ the localization at the prime ideal $U^\wedge_{\P_i}$ is a central simple $V^\wedge_{\P_i}$-algebra.   Any central simple $V^\wedge_{\P_i}$-algebra is a matrix algebra over a finite-dimensional division algebra over $V^\wedge_{\P_i}$, and as mentioned above,  \cite{BHM93} shows the Morita equivalence of   topological Hochschild homology.  Theorem \ref{main} is therefore proved by assembling, over all nontrivial prime ideals $\P\subset V$, the result:

\noindent {\bf{Theorem
\ref{localmainthm}}.}
Let $D$ be a finite-dimensional division algebra over $\Q_p$ and let $A$ be a maximal order in $D$.  Let $L$ be the center of $D$, and let $S$ be its valuation ring and $\F_S$ the residue field of $S$; we can write $S=R[\pi]/P(\pi)$, for $R$ unramified over $\Z_p$, $\pi$ a uniformizer of $S$, and $P$ an Eisenstein polynomial.   Assume that $D$ is of degree $n$ over $L$ (that is, of dimension $n^2$ over $L$).
Then
$$
\pi_*( \THH(A)^\wedge_p)\cong
\begin{cases}
S \oplus \F_S^{\oplus n-1} & *=0\\ 
S/(aP'(\pi)) & *=2a-1 > 0\\
\F_S^{\oplus n-1} &*=2a>0\\
 0 &  *<0.
 \end{cases}
$$

Note that Theorem \ref{localmainthm} shows that $\pi_*( \THH(A)^\wedge_p)$ is isomorphic to a copy of $\pi_*( \THH(S)^\wedge_p)$, which lives in dimension zero and in odd dimensions, summed with $\F_S^{\oplus n-1}$ in all even dimensions.  When $p$ does not divide $n$, the inclusion $S\hookrightarrow A$
induces a map  sending $\pi_*( \THH(S)^\wedge_p)$ isomorphically to the corresponding part of  $\pi_*( \THH(A)^\wedge_p)$, but if $p$ divides $n$, this is not so.

\medskip
We would like to thank Lars Hesselholt and Michael Larsen for useful conversations about the structure of maximal orders and their invariants, Vigleik Angeltveit for finding a problem in an earlier draft of this paper, and Peter May for his comments on an earlier draft  and  for his guidance.

\section{Basic set-up and Larsen's Hochschild homology calculation }\label{second}

In \cite[Section 3]{L95}, Larsen looks at the following set-up: $K$ is a complete local field with ring of integers $R$, $D$ is a division algebra over $K$, and $A$ is a maximal order in $D$.  He calculates $\HH_*^R(A)$ when the center of $D$, which is denoted $L$, is  separable and totally ramified over $K$.  The separability is not a problem since all the fields in question have characteristic $0$.  We are actually interested in calculating $\HH_*^{\Z_p}(A)$, and there is no reason that the center $L$ of $D$ should be totally ramified over $\Q_p$.  However, if we let $K$ be the maximal unramified extension of $\Q_p$ in $L$ and let $R$ be the valuation ring of $K$, by \cite[Theorem 3.1]{Li00}, we have a spectral sequence
\begin{equation}\label{overramified}
E^2_{s,t} =
\HH^R_s(A; \HH^{\Z_p}_t (R; A)) \Rightarrow \HH^{\Z_p} _{s+t}(A).
\end{equation}
Since $R$ is unramified over $\Z_p$, $\HH^{\Z_p}_*(R;A)$ consists only of $A$ in dimension $0$, so in fact the $E^2$ page is concentrated in the $0$'th row and we get an isomorphism $\HH_*^R(A) \cong \HH_*^{\Z_p}(A)$.  The map $\HH_*^{\Z_p} (A) \to \HH_*^{R}(A)$ which is induced by replacing $\otimes_{\Z_p}$    by $\otimes_R$ induces this isomorphism:  this can be seen by mapping the obviously collapsing spectral sequence
$E^2_{s,t} =
\HH^{\Z_p}_s(A; \HH^{\Z_p}_t (\Z_p; A)) \Rightarrow \HH^{\Z_p} _{s+t}(A)$ into the spectral sequence we are interested in by replacing the ${\Z_p}$'s by $R$'s.

So we let $K$ be the maximal unramified extension of $\Q_p$ in $L$ and get that $L$ is totally ramified over $K$.  In the beginning of \cite[Section 3]{L95}, it is shown that if the degree of $D$ over $L$ is $n$ (that is: the dimension of $D$ over $L$ is $n^2$) then there is a degree $n$ unramified extension $M$ of $L$ whose valuation ring we can call $T$, an element $x\in A$ so that $x^n=\pi$ for a uniformizer $\pi\in S$, and a generator $\sigma \in \Gal(M/L)\cong \Z/n\Z$ so that
$$D\cong M\oplus M\cdot x \oplus \cdots \oplus M\cdot x^{n-1}$$ 
and 
\begin{equation}\label{Astruc}
A \cong T\oplus T\cdot x\oplus T\cdot x^2\oplus\cdots \oplus T\cdot x^{n-1}
\end{equation}
and $mx=x\sigma(m)$ for all $m\in M$.  

Since $L$ is a totally ramified extension of $K$, on the valuation ring $S$ we get that the uniformizer $\pi$ of $S$ satisfies an Eisenstein polynomial $P(z)=z^d+p_{d-1} z^{d-1} +\cdots +p_1z+p_0$ where $d=[L:K]$,  and that
\begin{equation}\label{Sstruc}
S \cong R[\pi]/(P(\pi)).
\end{equation}

Under these conditions, Larsen constructs quasi-isomorphisms both ways (which we will be using later) between the reduced Hochschild complex of $A$ over $R$ and the small complex
\begin{equation}\label{small}
\xymatrix@1{0\quad  & \quad T \quad \ar[l] &\quad T \quad \ar[l]_{\pi(1-\sigma^{-1})} &\quad T \quad \ar[l]_{P'(\pi) \Tr} &\quad T \quad \ar[l]_{\pi(1-\sigma^{-1})} &\quad T\cdots \ar[l]_{P'(\pi)\Tr} },
\end{equation}
where $\Tr =\Tr_{T/S}=1+\sigma+\cdots+\sigma^{n-1}$.   Recall that the complex
\begin{equation}\label{resoln}
\xymatrix@1{ \cdots\quad &\quad  T \quad \ar[l] _{\Tr} &\quad T \quad \ar[l]_{1-\sigma^{-1}} &\quad T \quad \ar[l]_{ \Tr} &\quad T \quad \ar[l]_{1-\sigma^{-1}} &\quad T\cdots \ar[l]_{\Tr} }
\end{equation}
is exact: by  Nakayama's Lemma, it is enough to check exactness on the residue fields $\F_T=T/(\pi)$, $\F_S=S/(\pi)$.  There we know that the 
image of $\Tr_{\F_T/\F_S}$ must be equal to  $\F_S$: it is clearly contained in $\F_S$, it must be an $\F_S$-vector space, but it cannot be $\{0\}$ because then all the $p^{nd}$ elements of $\F_T$ would satisfy the polynomial $x+ x^{p^d}+x^{p^{2d}}+\cdots x^{p^{(n-1)d}} =0$ which has degree $p^{(n-1)d}$.  Once the image of $\Tr_{\F_T/\F_S}$ is known to be equal to $\F_S=\ker (1-\sigma^{-1})$, for dimension reasons the image of $1-\sigma^{-1}$ must also be all of $\ker(\Tr_{\F_T/\F_S})$.  

By  the quasi-isomorphism to the small complex in Equation (\ref{small}) given in the proof  of \cite[Theorem 3.5]{L95} (there is a misprint in the statement of the theorem there), we get the formula
\begin{equation}\label{HHlocal1}
\HH_*^R(A) \cong
\begin{cases}
T/\pi \ker( \Tr_{T/S} )& *=0\\ 
 S/ P'(\pi) S&*=2a-1>0\\
\ker( \Tr_{T/S})/ \pi \ker( \Tr_{T/S} )& *=2a > 0\end{cases}
\end{equation}
as  $S$-modules.  By the exactness of the complex in Equation (\ref{resoln}), we know that 
$S=\ker(1-\sigma^{-1})=\Tr_{T/S}(T)$ 
and $\ker(\Tr_{T/S}) = (1-\sigma^{-1}) (T)$, so we have short exact sequences 
\begin{equation}\label{SESs}
0\to S \to T\to \ker(\Tr_{T/S}) \to 0, \qquad 
0\to\ker(\Tr_{T/S})\to T\to S\to 0.
\end{equation}
  This second short exact sequence has to be split as a sequence of $S$-modules by $S$'s freeness; we also know that as an $S$-module, $T\cong S^{\oplus n}$.  That makes $\ker(\Tr_{T/S})$ a projective module over the discrete valuation ring $S$, hence free.  Since $\ker(\Tr_{T/S}) \oplus S \cong T\cong S^{\oplus n}$ we must have that $\ker(\Tr_{T/S}) \cong S^{\oplus n-1}$ as an $S$-module.   The fact that $\ker(\Tr_{T/S}) $ is free and in particular projective also forces the first short exact sequence to split.
There is however  no reason for the splittings $S\to T$ and  $\ker(\Tr_{T/S}) \to T$ of the two short exact sequences  above to be the obvious inclusions.  In fact, if $n>1$ the splitting $S\to T$ cannot be the obvious inclusion, since $\Tr_{T/S}$ restricted to $S$ is multiplication by $n$.
Nevertheless, understanding the split decompositions into free $S$-modules in both short exact sequences lets us   decompose the result in Equation (\ref{HHlocal1}) and write it more conveniently as $S$-module isomorphisms
\begin{equation}\label{HHlocal}
\HH_*^{\Z_p}(A)\cong \HH_*^R(A) \cong
\begin{cases}
S\oplus \F_S^{\oplus n-1}& *=0\\ 
 S/ P'(\pi) S\cong \HH_{2a-1}^{\Z _p}(S) &*=2a-1>0\\
 \F_S^{\oplus n-1} & *=2a > 0. \end{cases}
\end{equation}
Note that except for the $S$ in dimension zero, these Hochschild homology groups consist entirely of torsion.
%Note that this shows that $\HH_*^{\Z_p}(A)$ contains $S$ (the center of $A$) in dimension $0$, and beyond it only torsion.   We will see that the
%$S\cong \HH_0^{\Z_p}(S)$ inside  $\HH_0^{\Z_p}(A)$, as well as the $\HH_{2a-1}^{\Z _p}(S) \cong \HH_{2a-1}^{\Z _p}(A)$ are the images of
%$\HH_{*}^{\Z _p}(S)$ under the map induced by the inclusion $S\hookrightarrow A$.   Having this result for every localization  can be assembled into

\begin{cor}\label{globalcor}{\bf{(To Theorem 3.5 of \cite{L95}.})}
Let $B$ be a simple algebra over $\mathbb{Q}$, and let $U$ be a maximal order in it.  Let $C$ be the center of $B$, and let $V$ be its ring of integers.  For every nontrivial prime ideal $\P\subset V$, the completion $B^\wedge_\P$ is a central simple $C^\wedge_\P$-algebra, and so $B^\wedge_\P$ is isomorphic to a matrix ring on some central division algebra $D_\P$ over $C^\wedge_\P$, of degree $e_\P$.  Let $\F_\P= V/\P$.  Then we have $V$-module isomorphisms
\begin{equation}\label{globalHH}
\HH_*(U)\cong
\begin{cases}
V \oplus  \bigoplus_{\P\subset V \ \mathrm{prime}} \F_\P ^{\oplus e_\P-1} & *=0\\ 
 \HH_{2a-1} (V)  & *=2a-1> 0\\
  \bigoplus_{\P\subset V \ \mathrm{prime}} \F_\P ^{\oplus e_\P-1} & *=2a>0.\\
 \end{cases}
\end{equation}
\end{cor}

\begin{proof}
This is assembled from Equation (\ref{HHlocal}) over all nontrivial prime ideals $\P\subset V$.  Since $U$ is finitely generated as a module over the Dedekind domain $V$, so is each level of the Hochschild complex and so is each Hochschild homology group.  Thus each Hochschild homology group splits as a direct sum of a projective $V$-module with $V$-torsion groups.  The torsion groups split as a direct sum of $\P$-torsion over all nontrivial prime ideals $\P\subset V$, and these were calculated in \cite[Theorem 3.5]{L95} and are summed up here.  

If we had a nonzero projective $V$-module in $\HH_i(U)$ for some $i$, we would get nonzero projective modules in all the completions, so by examining Equation (\ref{HHlocal}), this happens only at $i=0$.  There we see, by looking at any of the completions, that this projective module has rank one.  So in fact the only part of this corollary that does not follow immediately from the calculation in \cite[Theorem 3.5]{L95} is the determination that the projective summand of $\HH_0(U)$ is isomorphic to $V$.  

It would suffice to produce a surjective $V$-linear map $\HH_0(U)\cong U/[U,U] \to V$, because then by $V$'s freeness over itself we would know that there is a section, so $\HH_0(U)$ would consist of a direct sum of $V$ with another module.  This complement would have to consist entirely of  torsion, because otherwise the  localization at any prime of $\HH_0(U)$ would have a higher rank free part than one copy of the localization of $V$, which is all there is in Larsen's result.

For any field $C$ and any finitely generated $C$-algebra $B$ we can define a trace map $\Trace_{B/C}: B\to C$ which assigns to every $b\in B$ the trace of the matrix describing left multiplication by $b$ as a $C$-linear function $B\to B$.  Clearly, $\Trace_{B/C}$ has to vanish on the commutators $[B,B]$.  If $B$ is a central simple $C$-algebra of degree $m$ (so of dimension $m^2$) over $C$, we can also look at the reduced trace 
$\Trd_{B/C} =\frac{1}{m} \Trace_{B/C}.$
We want to use the $B$ and $C$ given in the conditions of the corollary and have 
$$\Trd_{B/C}: U/[U,U] \to V$$
be the  required surjection,  In order for this to make sense,  we  need to verify that $\Trd_{B/C}(U)\subseteq V$ and that the image is all of $V$.  Both these conditions can be verified locally, so we will show that  $\Trd_{B^\wedge_\P/C^\wedge_\P}(U^\wedge_\P)$  has its image contained in  $V^\wedge_\P$ and in fact equal to all of $V^\wedge_\P$ for every nontrivial prime ideal $\P\subset V$.

After completing, using the notation discussed earlier in this section, we know  that 
$B^\wedge_\P \cong M_{i_\P\times i_\P}(D)$ for $D$ a central division algebra over $C^\wedge_\P$ of degree $e_\P$.  Counting dimensions, we get $m= i_\P e_\P$.  Since  the trace of an $( i_\P e_\P)\times ( i_\P e_\P)$ matrix can be calculated by first taking trace of an 
$i_\P\times i_\P$ matrix of $e_\P\times e_\P$ blocks and then taking trace of the resulting $e_\P\times e_\P$ block, 
\begin{equation}\label{splitting}
\Trd_{B^\wedge_\P/C^\wedge_\P} 
=\frac{1}{m} \Trace_{B^\wedge_\P/C^\wedge_\P}
= \frac{1}{e_\P} \Trace_{D/C^\wedge_\P} 
\circ \frac{1}{i_\P} \Trace_{B^\wedge_\P/D} .
\end{equation}
As in the introduction, since $U^\wedge_\P$ is a maximal order in $B^\wedge_\P \cong M_{i_\P\times i_\P}(D)$, 
$U^\wedge_\P \cong M_{i_\P\times i_\P}(A)$ for a maximal order $A$ in $D$.  Since $M_{i_\P\times i_\P}(A)$
as a representation of itself acting by left multiplication splits as a direct sum of $i_\P$ copies of $M_{i_\P\times i_\P}(A)$ acting on the columns
$M_{i_\P\times 1}(A)$, we get that $ \frac{1}{i_\P} \Trace_{B^\wedge_\P/D}$ sends $U^\wedge_\P $ to the image of the usual trace of $M_{i_\P\times i_\P}(A)$ acting on
$M_{i_\P\times 1}(A)$, which is  all of $A$.

So it remains to show that $ \frac{1}{e_\P} \Trace_{D/C^\wedge_\P} (A) \subseteq V^\wedge_\P$ and that it is in fact all of $V^\wedge_\P$.  To show the containment, we find a finite extension $\tilde C$ of $C^\wedge_\P$ over which $D$ splits, that is: so that $\tilde C\otimes_{C^\wedge_\P} D \cong M_{e_\P\times e_\P}(\tilde C)$.  Let $\tilde V$ be $\tilde C$'s valuation ring.  By the previous argument for matrix rings,
 $$ \frac{1}{e_\P} \Trace_{(\tilde C\otimes_{C^\wedge_\P} D)/\tilde C} (\tilde V \otimes_{V^\wedge_\P}A) = \tilde V,$$
 and of course 
  $$ \frac{1}{e_\P} \Trace_{(\tilde C\otimes_{C^\wedge_\P} D)/\tilde C} (D) =
\frac{1}{e_\P} \Trace_{D/C^\wedge_\P} (D) \subseteq D.$$
Therefore, for the intersection of these two we have
 $$ \frac{1}{e_\P} \Trace_{(\tilde C\otimes_{C^\wedge_\P} D)/\tilde C} (A) =\frac{1}{e_\P} \Trace_{D/C^\wedge_\P}(A)  \subseteq \tilde V\cap D = V^\wedge_\P.$$
To see that the image of $A$ under $ \frac{1}{e_\P} \Trace_{D/C^\wedge_\P}$ is indeed all of $V^\wedge_\P$, it is enough to show that 
$ \frac{1}{e_\P} \Trace_{D/C^\wedge_\P}(T)=V^\wedge_\P$ for the $T\subseteq A$ defined above Equation (\ref{Astruc}).  By that equation, $A\cong T^{\oplus e_\P}$ as a $T$-module (since $n= e_\P$ in this context).  Therefore, on $T$, 
$ \frac{1}{e_\P} \Trace_{D/C^\wedge_\P}=\Trace_{M/C^\wedge_\P}$, where $M$ is the  field of fractions of $T$.  But the trace of $M$ over $L=C^\wedge_\P$ is the same as the 
$\Tr =\Tr_{T/S}=1+\sigma+\cdots+\sigma^{n-1}$ defined above Equation (\ref{resoln}).   By the exactness  in Equation (\ref{resoln}),  $\Tr_{T/S}(T)=\ker(1-\sigma^{-1}) = S =V^\wedge_\P$.
\end{proof}

We conclude this section with a lemma that we will use later, elaborating on the complex (\ref{small}), and a few consequences:
\begin{lem}\label{themap}
If we calculate   $\HH^R_*(A)$ using the complex (\ref{small}), then the following map of complexes induces on homology the map that the inclusion $T\hookrightarrow A$ induces
on Hochschild homology:
$$\xymatrix{
0& \ar[l] T \ar[d]^{=} \quad &
 \ar[l]_0 \ar[d] ^\Tr T \quad &
 \ar[l] _{P'(\pi)} T \ar[d]^{=} \quad &
  \ar[l]_0 \ar[d] ^\Tr T\quad  &
 \ar[l] _{P'(\pi)} T\quad  \ar[d]^{=} &
 \ar[l]_0 \cdots
 \\
0&
\ar[l] T\quad &
 \ar[l]_{\pi(1-\sigma^{-1})}  T\quad  &
  \ar[l]_{P'(\pi) \Tr} T\quad  &
  \ar[l]_{\pi(1-\sigma^{-1})}  T\quad  &
  \ar[l]_{P'(\pi) \Tr} T\quad  &
  \ar[l]_{\pi(1-\sigma^{-1})}   \cdots
}
$$
\end{lem}
\begin{proof}
To see that the homology of the complex in the top row is $\HH^R_*(T)$, consider the tensored-down version of the resolution (1.6.1) in \cite{LL92} for $S=R[\pi]/P(\pi)$.  Since  $T$ is free and unramified  over $S$, $\HH^R _*(T) \cong \HH^R_*(S) \otimes_S T$.  Therefore, if we have a complex of free $S$-modules calculating $ \HH^R_*(S)$, it can be tensored over $S$ with $T$ to give a complex calculating $\HH^R_*(T)$.  Thist is exactly what the top row is.

The proof of the lemma is a direct calculation using the weak equivalences of \cite{LL92} and  \cite{L95}: take the weak equivalence from the small complex calculating $ \HH^R_*(S)$ into the standard Hochschild complex of $S$ over $R$ that is given in \cite[Equation (1.8.6)]{LL92}.
Tensor the $0$'th $S$ coordinate over $S$ with $T$ to obtain a map from the complex in the top row of our diagram  into the standard Hochschild  complex of $T$ over $R$.  Since the standard Hochschild complex is a functor, this includes in the obvious way into the standard Hochschild complex of $A$ over $R$.  The standard Hochschild complex maps to the reduced Hochschild complex, and the map $\pi_*$ defined below \cite[Equation (3.8.1)]{L95} can be applied.  The resulting map of complexes can readily be seen to have period 2.  The map $\pi_0$ sends elements of $T$ to themselves.  To understand $\pi_1$, let  $U$ be the valuation ring in the maximal unramified extension of $R$'s field of fractions in $T$'s field of fractions (see the diagram below \cite[Equation (3.2.1)]{L95}).   The equation 
$\pi_1(u_1 \pi^{\ell_1}\otimes  \pi^{\ell_2}) = \Tr(u_1)  \pi^{\ell_1+\ell_2 -1}$ for $u_1\in U$, $\ell_2>0$ means that an element $u_1 \pi^{\ell_1}\otimes  \pi^{\ell_2}$ which corresponds to $u_1 \pi^{\ell_1+\ell_2 -1}$ in dimension $1$ in the top row of our complex maps to $ \Tr(u_1)  \pi^{\ell_1+\ell_2 -1}$ in the bottom row.  Since $T=U[\pi]/P(\pi)$, these elements span $T$.
\end{proof}

\begin{cor}\label{tamedivisionalgeba}
If the degree $n$ of $D$ over $L$ is not divisible by $p$ then the inclusion $S\hookrightarrow A$ induces the map of
$\HH_*^R(S) \cong
\begin{cases}
S & *=0\\ 
 S/ P'(\pi) S&*=2a-1>0\\
0& *=2a > 0.\end{cases}$
into the corresponding parts of $ \HH_*^R(A) \cong
\begin{cases}
S\oplus \F_S^{\oplus n-1}& *=0\\ 
 S/ P'(\pi) S &*=2a-1>0\\
 \F_S^{\oplus n-1} & *=2a > 0. \end{cases}
$
by the inclusion in dimension zero and multiplication by the unit $n\in S$ in odd dimensions.
\end{cor}
 \begin{proof}
 If $p\nmid n$, we can split the second short exact sequence
 $$
0\to\ker(\Tr_{T/S})\to T\to S\to 0$$ in
(\ref {SESs})
 by $1/n$ times the inclusion $ S\hookrightarrow T$ and get a decomposition of $S$-modules $T\cong \ker(\Tr_{T/S})\oplus S$ that works for both short exact sequences in (\ref {SESs}) in terms of breaking $T$ up into two $S$-submodules, even if the maps are not exactly the standard inclusions---they differ from that only by multiplication by a unit of $S$.  Recall that the top complex in Lemma \ref{themap} was obtained by taking the complex of $S$-modules 
$\xymatrix{
0& \ar[l] S  &
 \ar[l]_0 S &
 \ar[l] _{P'(\pi)} S &
  \ar[l]_0S &
 \ar[l] _{P'(\pi)} S&
 \ar[l]_0 \cdots}$
 and tensoring it over $S$ with $T\cong \ker(\Tr_{T/S})\oplus S$.  It is therefore isomorphic to the direct sum of the complex
 \begin{equation}\label{firstcomplex}
 \xymatrix{
0& \ar[l] S  &
 \ar[l]_0 S &
 \ar[l] _{P'(\pi)} S &
  \ar[l]_0S &
 \ar[l] _{P'(\pi)} S&
 \ar[l]_0 \cdots,}
 \end{equation}
 whose homology is the image of $\HH_*^R(S)$ in $\HH_*^R(T)$, with the complex
  \begin{equation}\label{secondcomplex}
 \xymatrix{
0& \ar[l] \ker(\Tr_{T/S})  &
 \ar[l]_0 \ker(\Tr_{T/S}) &
 \ar[l] _{P'(\pi)} \ker(\Tr_{T/S}) &
   \ar[l]_0 \cdots}.
 \end{equation}
 Similarly, the bottom complex in Lemma  \ref{themap} can be split into the direct sum of the complex
 \begin{equation}\label{thirdcomplex}
 \xymatrix{
0& \ar[l] S  &
 \ar[l]_0 S &
 \ar[l] _{nP'(\pi)} S &
  \ar[l]_0S &
 \ar[l] _{nP'(\pi)} S&
 \ar[l]_{\ \ 0} \cdots,}
 \end{equation}
 whose homology accounts for the copy of $\HH_*^R(S)$ inside $\HH_*^R(A)$, and the complex
 \begin{equation}\label{fourthcomplex}
 \xymatrix{
0& \ar[l] \ker(\Tr_{T/S}) \  \ &
 \ar[l]_{\pi(1-\sigma^{-1})} \ \ker(\Tr_{T/S}) &
 \ar[l] _{0} \ker(\Tr_{T/S}) \ \ \ \ &
   \ar[l]_{\pi(1-\sigma^{-1})} \cdots}.
 \end{equation} 
Note that $(1-\sigma^{-1}):\  \ker(\Tr_{T/S})\to \ker(\Tr_{T/S})$ is an isomorphism because we know that $(1-\sigma^{-1})(T)= \ker(\Tr_{T/S})$ and that $\ker (1-\sigma^{-1})$ is exactly the the image of the standard inclusion of $S$ in $T$; thus the homology of the complex (\ref{fourthcomplex}) accounts for the $ \F_S^{\oplus n-1} $'s in even dimensions in $\HH_*^R(A)$.

Lemma \ref{themap} explains to us exactly what the inclusion $T\hookrightarrow A$ does to these parts: the complex (\ref{firstcomplex}) maps by the identity in even dimensions and multiplication by the unit $n$ in odd dimensions onto the complex (\ref{thirdcomplex}) while the complex (\ref{secondcomplex}) maps into the complex
 (\ref{fourthcomplex}) by the identity in even dimensions and the zero map in odd dimensions.  Note however that most of the homology of the complex  (\ref{secondcomplex}) lies in odd dimensions, and therefore goes to zero: the only even-dimensional homology is $ \ker(\Tr_{T/S})$ in dimension zero which goes by the obvious quotient map to  $\ker(\Tr_{T/S})/\pi \ker(\Tr_{T/S})\cong  \F_S^{\oplus n-1}$.
 \end{proof}
 
Repeating the above argument after tensoring both complexes in Lemma  \ref{themap} with $\F_p$, we get
\begin{cor}\label{tamedivisionalgebamodp}
If the degree $n$ of $D$ over $L$ is not divisible by $p$ then the inclusion $S\hookrightarrow A$ induces an embedding of
$\HH_*(S/(p)) $ as a direct summand in $\HH_*(A/(p)) $, with the complementary direct summand  consisting of a copy of 
$\F_S^{\oplus n-1}$ in every dimension  $*\geq 0$.
\end{cor}

In the general case, when  the degree $n$ of $D$ over $L$ might be divisible by $p$, we can say less, but we can still deduce the following two corollaries:

\begin{cor}\label{wilddivisionalgeba}
The inclusion $T\hookrightarrow A$ induces a map of
from $\HH_*^R(T)$
to $ \HH_*^R(A)$
which is the quotient map $T\to T/\pi \Tr_{T/S}(T)$ in dimension zero, a surjection in odd dimensions, and the zero map in positive even dimensions.
\end{cor}

\begin{proof}
In odd dimensions, $\Tr_{T/S}$ in the vertical arrows in the diagram in Lemma \ref{themap} 
sends $T$ onto $\ker(\pi (1-\sigma^{-1})) = \ker(1-\sigma^{-1})= \im(\Tr_{T/S})$.  In even dimensions, we have that for $a>0$, $\HH_{2a}^R(T)\cong 0$.
\end{proof}

\begin{cor}\label{wilddivisionalgebamodp}
The inclusion $T/(p) \hookrightarrow A/(p)$ induces a map 
%on Hochschild homology 
$\HH_*(T/(p))\to \HH_*(A/(p))\cong \HH_*(S/(p)) \oplus \F_S^{\oplus n-1}$
which in dimension zero  is the quotient map  $T/(p)\to T/(p, \pi \Tr_{T/S}(T))$
and for all $*> 0$ subjects onto the first summand $\HH_*(S/(p))$.   \end{cor}

\begin{proof}
Since $P$ is an Eisenstein polynomial, $P(\pi)\equiv \pi^d\equiv 0$ and $P'(\pi)\equiv d\pi^{d-1}$ modulo $p$.   After tensoring both complexes in Lemma  \ref{themap} with $\F_p$, we see that in odd dimensions in the bottom row of the diagram in Lemma  \ref{themap}, $\ker(\pi (1-\sigma^{-1})) $ becomes the direct sum of  
$\im(\Tr_{T/S})$ with $\pi^{d-1}$ times its complement in $T/(p)$.  The former is the image of the vertical map and its quotient by the incoming boundary map is the $\HH_{2a-1}(S/(p))$ part of  $\HH_{2a-1}(A/(p))$.  The latter gives the copy of $ \F_S^{\oplus n-1}$.  

In even dimensions $2a>0$
 in the bottom row of the diagram in Lemma  \ref{themap}, $\ker(d\pi^{d-1}\Tr_{T/S})$ becomes the direct sum of
$\ker(\Tr_{T/S})$ and
of $\Ann_{S/(p)}( d\pi^{d-1})$. The  quotient of $\ker(\Tr_{T/S})$ by the incoming boundary map is the  copy of $ \F_S^{\oplus n-1}$  in  $\HH_{2a}(A/(p))$.  The copy 
of $\Ann_{S/(p)}( d\pi^{d-1})$  in  the complement of  $\ker(\Tr_{T/S})$ in $T/(p)$ forms  the $\HH_{2a}(S/(p))$ in  $\HH_{2a}(A/(p))$.
\end{proof}

\section{The calculation of $\THH_*( A/(p))$ for $A$ a maximal order in a division algebra over $\Q_p$ }\label{third}

In the introduction, we explained how we can find all the torsion in $\THH_*(U)$ by looking at  $\THH_*(A)$  for appropriate maximal orders $A$ in division algebras over $\Q_p$.  Our approach to the calculation
of  $\THH_*(A)$  is to start with finding 
$\THH_*(A/(p))$  for such algebras.  Then we use that and 
the Brun Spectral Sequence from \cite[Theorem 3.3]{LM00} to calculate $\THH_*(A;A/(p))\cong \pi_*(\THH(A); \F_p)$, which gives us the rank of the $p$-torsion in each dimension, and finally we analyze the order of the torsion.

\begin{prop} \label{modp} Let $A $ be the maximal order in a division algebra over $\mathbb{Q}_p$ with center $S$.
Then there is an isomorphism of $\THH_*(\F_S)$-modules 
\[
\THH_*(A /(p))\cong \THH_*(\F_S)\otimes_{\F_S} \HH^{\F_S} _*(A /(p)).
\]
which can also be viewed as 
an isomorphism of $\THH_*(\F_p)\otimes \F_S$-modules 
\[
\THH_*(A /(p))\cong \THH_*(\F_p)\otimes \HH_*(A /(p)).
\]

\end{prop}

\begin{proof}
Using the notation of the previous section, recall that $T$ is a degree $n$ unramified extension of $S$, so $T$ has the same uniformizer $\pi$ that satisfies a degree $d$ Eisenstein polynomial $P$ over $R$, and $T$ is generated over $S$ by some element whose reduction modulo $(\pi)$ generates $\F_T$ over $\F_S$.  
Therefore $T/(p)\cong \F_T[\pi]/(\pi^d)$.
By our description of $A$ from Equation (\ref{Astruc}), this means 
\[
A /(p)\cong \ \F_T[\pi]/(\pi^d)\oplus  \F_T[\pi]/(\pi^d)\cdot x \oplus \cdots\oplus  \F_T[\pi]/(\pi^d)\cdot x^{n-1},
\]
where $x^n=0$, $\pi x=x\pi$,  and $fx=x\sigma(f)$ for all $f\in\F_T$ for a generator $\sigma$ of $\Gal(M/N)\cong \Gal(\F_T/\F_S) \cong\Z/n\Z$. 

By \cite[Corollary 3.3]{Li00} applied to the $\F_S$-algebra $A/(p)$ we get a spectral sequence of $\THH_*(\F_S)$-algebras
\[
E_{r,s}^2=\HH^{\F_S} _r(A /(p); \THH_s(\F_S; A /(p))) \Rightarrow \THH_{r+s}(A /(p)).
\]
Since all $\F_S$-modules are free, 
\[
E_{r,s}^2 \cong \HH^{\F_S} _{r}(A /(p))\otimes_{\F_S}  \THH_s(\F_S).
\]
The claim is that this spectral sequence  collapses at $E^2$, and the  $\THH_*(\F_S)$-algebra structure on the $E^2=E^\infty$ term is the correct one.  

The first formulation of the proposition implies the second since $\F_S$ is unramified over $\F_p$ and so
$\THH_*(\F_S) \cong \THH_*(\F_p)\otimes \F_S$ and also (by the same argument as that in Equation (\ref{overramified})), $\HH^{\F_S} _*(A /(p))\cong  \HH_*(A /(p))$.

\begin{dfn}
Let $k$ be a field.  We will say that a  unital $k$-algebra $C$ is  weakly monoidal if it has a basis $\mathscr{B}$ over $k$ so that $\mathscr{B}\cup\{0\}$ is closed under multiplication.
\end{dfn}

Note that this is weaker than the definition of a pointed monoid algebra in \cite[Section 7.1]{HM97}, which also requires that the unit $1\in C$ should be in $\mathscr{B}$. 
However, if $C$ is weakly monoidal, we can still define the cyclic nerve $N^{cy}(\mathscr{B}_+)$, for  $\mathscr{B}_+
= \mathscr{B}\cup\{0\}$, as in \cite[Section 7.1]{HM97}.  It is no longer a simplicial set, just a semisimplicial one.  We can map the suspension spectrum of its semisimplicial realization $\Sigma^{\infty} N^{cy}(\mathscr{B}_+)$ into the `fat' realization of $\THH(C)$: the one that uses only its semisimplicial structure and not the degeneracies.  We can do this for example in B\"okstedt's model for $\THH$ of functors with smash product, which assigns to each level in the spectrum a simplicial space, where the simplicial structure maps and the spectrum structure maps commute; by the theory of simplicial spaces, at each level of the spectrum, if we use a `fat' realization ignoring the degeneracies, we get something homotopy equivalent to the usual realization.  
We still get that
$$\HH^k_*(C)=  \HH^k_*(k[\mathscr{B}])\cong \tilde H_*( N^{cy}(\mathscr{B}_+); k) \cong \pi_*( Hk \wedge N^{cy}(\mathscr{B}_+)).$$
We can also map $Hk \to \THH(C)$ by using the unit map of $C$ and including into the $0$-skeleton.  Since we have a product $\THH(k) \wedge \THH(C) \to \THH(C)$ (because $k$ is in the center of $C$, even if the multiplication of $\mathscr{B}$ is not commutative),   we can map
\begin{equation*} Hk \wedge N^{cy}(\mathscr{B}_+) \to \THH(C) \to \HH^{\mathbb{Z}} (C) \to \HH^k(C).
\end{equation*}
Here the map before last is the  linearization map, and the last map is induced by taking tensor products over $k$ rather than over ${\mathbb{Z}}$. The homotopy groups of the first spectrum and the last spectrum are both $\HH_*^k(C)$, and the composition induces an isomorphism between the two.
On the spectral sequence
 \[ E_{r,s}^2= \HH^k _r(C){\otimes}_k \THH_s(k) \Rightarrow \THH_{r+s}(C),
\]
 linearization and tensoring over $k$ induce the component map $\THH_*(k)\to k$ on the columns.   So if we know that our composition map, which passes through $\THH(C)$, induces an isomorphism on $\HH^k_*(C)$, the spectral sequence differentials $d^r$  have to vanish on the $0$'th row for all $r\geq2$.  Since $\THH_*(k)$ sits in the $0$'th column of a first quadrant spectral sequence, all the $d^r$,  $r\geq 2$ must vanish on it as well.  We recall that our spectral sequence respects multiplication by  $\THH_*(k)$, and deduce from the vanishing of its differentials on $\THH_*(k)$ and $\HH^k_*(C)$  that it
collapses at $E^2$.  We can also deduce that the $E^\infty = E^2$-term has the correct $\THH_*(k)$-algebra structure since it is maximally nontrivial as a $\THH_*(k)$-algebra.  We get that for $C$ weakly monoidal over a field $k$, 
$$\THH_*(C)\cong \HH^k _*(C){\otimes}_k \THH_*(k) .$$

Thus Proposition \ref{modp} would follow directly if we knew that $A /(p)$ was weakly monoidal over 
$\F_S$, but we do not know that.  We have the following, instead:

\begin{lem}
$\F_T\otimes_{\F_S} A /(p)$ is weakly monoidal over $\F_T$.
\end{lem}
\begin{proof}
By our decomposition of $A /(p)$, we know that 
\[
A /(p)\cong\F_T[\pi]/(\pi^d)\oplus\F_T[\pi]/(\pi^d)\cdot x \oplus \cdots\oplus\F_T[\pi]/(\pi^d)\cdot x^{n-1}
\]
with $x^n=0$ and $fx=x\sigma(f)$ for all $f\in\mathbb{F}_{p^n}$ and $\pi x=x\pi$.
Thus 
\begin{align*}
& \F_T  \otimes_{\F_S}  A /(p)  \cong
\\ &
 (\F_T\otimes_{\F_S} \F_T[\pi]/(\pi^d))
 \oplus (\F_T\otimes_{\F_S} \F_T[\pi]/(\pi^d)) x \oplus \cdots \oplus  (\F_T\otimes_{\F_S} \F_T[\pi]/(\pi^d)))x^{n-1}
\end{align*}
where $x$ commutes with the first tensor factor of $\mathbb{F}_{T}$, but not with the second one, in each summand.

The map
$\varphi:\ \F_T\otimes_{\F_S} \F_T \to \bigoplus_{i=1}^n \F_T$
given by
\[
\varphi( a\otimes b) =(ab, a\sigma(b), a\sigma^2(b), \cdots, a\sigma^{n-1}(b))
\]
for the  generator $\sigma\in \Gal(\F_T/\F_S)$ that we have been working with is an algebra isomorphism: it is obviously  a homomorphism. Since its domain and range have the same number of elements, it suffices to show that it is injective.
To show injectivity, let $b_0$ be a primitive element of $\F_T$.  Then the elements $1, b_0,\ldots, b_0^{n-1}$ span  $\F_T$ over  $\F_S$ and $\sigma^j(b_0)$,
 $0\leq j \leq n-1$ are the $n$ distinct roots of $b_0$'s minimal polynomial.  If  $\varphi(\sum_{i=0}^{n-1} a_i \otimes b_0^i )= (0, \ldots, 0)$, we must have $\sum_{i=0}^{n-1} a_i \sigma^j( b_0^i )=
\sum_{i=0}^{n-1} a_i (\sigma ^j(b_0))^i  = 0$ for $0\leq j\leq n-1$.  But then we have a degree $n-1$ polynomial over a field with $n$ distinct roots, which is impossible unless the polynomial is identically zero.

Thus we can identify $\F_T\otimes_{\F_S} A /(p)$ with
\[
\bigoplus_{i=1}^n\F_T[\pi]/(\pi^d)\oplus \bigoplus_{i=1}^n\F_T[\pi]/(\pi^d) \cdot x\oplus \cdots \oplus \bigoplus_{i=1}^n\F_T[\pi]/(\pi^d)\cdot x^{n-1}.
\]
 Let $e_1, \cdots, e_n$ be the standard basis for $\bigoplus_{i=1}^n \F_T$ over $\F_T$, where we take the indices of the $e_i$ to be in $\mathbb{Z}/n$.
 Then the relation
\[
(a\otimes b)x=x(a\otimes \sigma(b))
\]
implies that 
\[
(ab, a\sigma(b), \cdots, a\sigma^{n-1}(b))x=x(a\sigma(b), a\sigma^2(b), \cdots, a\sigma^{n-1}(b), ab),
\]
that is:  $e_ix=xe_{i-1}$ for all $i\in \mathbb{Z}/n$.

The $\F_T$-basis we take for $\F_T\otimes_{\F_S}  A /(p)$ is 
$$\mathscr{B}=\{e_i \pi^j x^k:  1\leq i\leq n, 0\leq j\leq d-1, 0\leq k\leq n-1\}$$
 with the multiplication 
\begin{align*}
(e_ix^j)\cdot(e_kx^l) & = e_i(x^je_k)x^l 
 = e_ie_{k+j}x^jx^l\\
 & = \left\{\begin{array}{ll} e_ix^{j+l} & i\equiv k+j\mod n\text{, and } j+l<n\\ 0 & \text{otherwise}\end{array}\right.
\end{align*}
and powers of $\pi$ commuting with everything.
\end{proof}

By this lemma and the discussion about weakly monoidal algebras preceding it, we get that the spectral sequence 
\begin{align}\label{Etwopn}
E_{r,s}^2 (\F_T\otimes_{\F_S}A /(p))
 = \HH^{\F_T} _r
  (\F_T\otimes_{\F_S}   A /(p))  &{\otimes}_{\F_T}  \THH_s(\F_T)  \\ & \Rightarrow \THH_{r+s}(\F_T\otimes_{\F_S}  A /(p)) \nonumber
\end{align}
collapses at $E^2$.
We observe that the spectral sequence (\ref{Etwopn})
is obtained from the analogous spectral sequence
\begin{equation}\label{Etwop}
 E_{r,s}^2( A /(p)) = \HH_r^{\F_S} (A /(p)) {\otimes}_{\F_S}  \THH_s(\F_S) \Rightarrow \THH_{r+s}( A /(p))
 \end{equation}
 by tensoring it over $\F_S$ with  $\F_T$: Since  $\F_T$ is flat over $\F_S$, 
$ \HH_*^{\F_T} (\F_T\otimes_{\F_S}  A /(p)) \cong \F_T \otimes_{\F_S} 
 \HH_*^{\F_S} (A /(p))$, and since $\F_T$ is \'etale over $\F_S$, 
 $ \THH_*(\F_T) \cong \F_T \otimes_{\F_S}   \THH_*(\F_S)$.

Note that tensoring with $\F_T$ is faithfully flat for $\F_S$-modules, and the $\F_T$ in the $E^2$ term in (\ref{Etwopn}) sits in bidegree $(0,0)$ so all differentials have to vanish on it for dimension reasons. We get that the spectral sequence (\ref{Etwopn}) calculating $\THH_*(\F_T\otimes_{\F_S}  A /(p))$ collapses at $E^2$ if and only if the spectral sequence (\ref{Etwop}) calculating $\THH_*(A /(p))$ does.  But the spectral sequence (\ref{Etwopn}) does collapse.  So
\[
E_{r, s}^\infty(A /(p))\cong E_{r, s}^2(A /(p))\cong \HH_r^{\F_S} (A /(p))\otimes_{\F_S} \THH_s(\F_S).
\]
The only remaining thing to check in order to complete the proof of  Proposition \ref{modp} is that $\THH_*(A /(p))\cong \HH_*^{\F_S} (A /(p))\otimes_{\F_S} \THH_*(\F_S)$ also as a $ \THH_*(\F_S)$-module.
But that follows since the multiplication by $ \THH_*(\F_S)$ is maximally nontrivial in the $E^\infty$-term.
\end{proof}

\section{The calculation of $\THH_*( A, A/(p))$ for $A$ a maximal order in a division algebra over $\Q_p$ }\label{fourth}
Our calculation will depend on whether $A$'s center $S$ is wildly ramified over $\Z_p$ or not.   Recall from Equation (\ref{Sstruc}) that if $R$ denotes the valuation ring of the maximal unramified extension of $\mathbb{Q}_p$ inside the center of the division algebra in question, we have $S\cong R[\pi]/(P(\pi))$ for $P$ an Eisenstein polynomial of degree $d$.  The ramification is wild when $p|d$, and otherwise it is tame.  The unramified case can be viewed as the case $d=1$.  Since $P$ is an Eisenstein polynomial,  if $\F_S=R/(p)=S/(\pi)$ is the residue field of $R$ and $S$, we get that $S/(p) \cong \F_S[\pi]/(\pi^d)$.

From \cite[Theorem 5.1]{LM00}, the local version of Theorem 1.1 that was quoted in the introduction, we get $S$-module isomorphisms
\begin{equation}\label{THHS}
\THH_i(S) \cong
\begin{cases}
S& i=0\\ 
 S/ (a P'(\pi)) &i=2a-1>0\\
0 & i=2a > 0. \end{cases}
\end{equation}
 Using the Universal Coefficient Theorem over $S$, we can  calculate $\THH_*(S; S/(p))$: modulo $p$,
 $P'(\pi)$ reduces to the same thing as $d\pi^{d-1}$  which is divisible by $p$ if $p|d$ but divides $p$  if $p\nmid d$.  So if $p|d$, for all for $i\geq 0$
 \begin{equation}\label{THHSSpwild}
\THH_i(S, S/(p))\cong  S/(p) \cong \F_S[\pi]/(\pi^d)
 \end{equation}
and if $p\nmid d$ (which includes the case of $S$ unramified over $\Z_p$, where $R=S$ and $d=1$),
\begin{equation}\label{THHSSptame}
\THH_i(S, S/(p)) \cong
\begin{cases}
S/(p) \cong \F_S[\pi]/(\pi^d)& i=0\  \mathrm{or} \ 2pk-1 \  \mathrm{or} \ 2pk, \ k\geq 1\\ 
S/( \pi^{d-1}) \cong \F_S[\pi]/(\pi^{d-1})& i=2k\  \mathrm{or}\ 2k-1,\ k\geq 0, \  p\nmid k. \end{cases}
\end{equation}

\medskip
\begin{thm} \label{onemodponenot}
Let $A$ be the maximal order in a division algebra over $\mathbb{Q}_p$ of degree $n$ (so of dimension $n^2$) over its center.  Let $S$ be the valuation ring of the center of the division algebra,
and let $\F_S$ be its residue field.
Then for all $i\geq 0$, there is an isomorphism of $S$-modules
\[
\THH_i(A, A/(p))\cong \THH_i(S, S/(p)) \oplus \F_S^{\oplus n-1}.
\]
When $p\nmid n$, the $\THH_i(S, S/(p))$ in this decomposition is  the isomorphic image of the map $\THH_i(S, S/(p))
\to \THH_i(A, A/(p))$ induced by the inclusion $S\hookrightarrow A$.  When $p|n$, that is not true but we do have that for $T$ the valuation ring of a degree $n$ unramified extension of the center that exists inside $A$, the inclusion  $T\hookrightarrow A$ induces a map $\THH_i(T, T/(p))\to \THH_i(A, A/(p))$ which in terms of the decomposition above, surjects onto the first factor.
The notation $\THH_i(S, S/(p))$ in the decomposition should be viewed when $p|n$ as shorthand notation for the different cases
(\ref{THHSSpwild}) and (\ref{THHSSptame}) above.
\end{thm}
\begin{proof}
We will use the Brun spectral sequence from \cite[Theorem 3.3]{LM00} which is associated to the reduction map $ A\to A/(p)$.  It is of the form 
\begin{equation}\label{Brun}
E^2_{r, s} =\THH_r(A/(p), \Tor ^{A}_s(A/(p), A/(p)) )\Rightarrow \THH_{r+s} 
(A, A/(p))
\end{equation}
which, since  $\Tor ^{A}_*(A/(p), A/(p))$ is just $ A/(p)$ in dimensions $0$ and $1$, consists of two rows, each isomorphic to $\THH_*(A/(p))$.  If we let $\tau$ denote a generator of $\Tor ^{A}_1(A/(p), A/(p))$ in bidegree $(0,1)$ in the spectral sequence, then
$E^2_{*,*} \cong \F_p[\tau]/\tau^2 \otimes \THH_*(A/(p))$, and by Proposition \ref{modp} and B\"{o}kstedt's calculation \cite{Bo} of $\THH_*(\F_p)\cong \F_p[u]$ for a $2$-dimensional generator $u$,
$$E^2_{*,*} \cong \F_p[\tau]/\tau^2 \otimes  \F_p[u]  \otimes \HH_*(A/(p)).$$

As explained in Equation (\ref{HHlocal}) above, \cite{L95} shows that
\begin{equation*}
\HH_i^{\Z_p}(A)\cong \HH_i^R(A) \cong
\begin{cases}
S\oplus \F_S^{\oplus n-1}& i=0\\ 
 S/ (P'(\pi)) \cong \HH_{2a-1}^{\Z _p}(S) &i=2a-1>0\\
 \F_S^{\oplus n-1} & i=2a > 0. \end{cases}
\end{equation*}
The relative Hochschild homology $\HH_*  ^{\mathbb{Z}_p} (A)$ is actually smaller than $\HH_*^\mathbb{Z}(A) $.  We want to use the above result to get $\HH_*(A/(p))$, and for that the difference does not matter: The Hochschild complex for the quotient ring $A/(p)$ is just $\F_p$ tensored with  the Hochschild complex for the ring $A$, and once we tensor with  $\F_p$,
$A^{\otimes _{\Z_p} (i+1)}  \otimes \F_p   \cong A^{ {\otimes _\Z} (i+1)} \otimes \F_p$ for all $i\geq 0$.

Applying   the Universal Coefficient Theorem to the Hochschild complex of $A$, observing that everything except for the $S$ in dimension zero in Equation (\ref{HHlocal}) is torsion of order which is a power of $p$, we therefore get
 \[
\HH_i(A/(p)) \cong \left\{\begin{array}{ll}  S/(p) \oplus \F_S^ {\oplus (n-1)} & \text{ if } i=0 \\ \F_S^ {\oplus (n-1)} \oplus S/ (p, P'(\pi))  & \text{ if } i>0.\end{array}\right.  
\]
We would also like to understand this homology  in terms of generators, so we view the Hochschild homology of $A/(p)$  in terms of
the small complex in Equation (\ref{small}).  In the proof of \cite[Theorem 3.5]{L95}, Larsen shows that the reduced Hochschild complex of $A$ is a direct sum of the small complex (\ref{small}) and another complex, and that the inclusion induces a quasi-isomorphism.  Thus, the other summand must be acyclic.  If we tensor everything with $\F_p$, the direct sum decomposition will continue to hold and the second complex will continue to be acyclic, so the modulo $p$ version of Equation  (\ref{small}) calculates $\HH_*(A/(p))$:
\begin{equation}\label{psmall}
\xymatrix@1{0\quad  & \quad T/(p) \quad \ar[l] &\quad T/(p) \quad \ar[l]_{\bar\pi(1-\bar\sigma^{-1})} &\quad T/(p) \quad \ar[l]_{{\overline{P'(\pi)}} \bar \Tr} &\quad T/(p)\cdots  \quad \ar[l]_{\bar\pi(1-\bar\sigma^{-1})}} ,
\end{equation}
where we use bars to denote the reductions of the elements and functions modulo $p$.  Moreover, the reduction modulo $p$ of the original quasi-isomorphism will continue to be a quasi-isomorphism.

\begin{prop}\label{keylemma}
In the spectral sequence (\ref{Brun}), $d^2$ sends
$\HH_i(A/(p)) \subset E^2_{i, 0}$ to
$\tau \HH_{i-2}(A/(p)) \subset E^2_{i-2, 1}$ for all $i\geq 2$, and in terms of the complex (\ref{small}),
$$d^2([m]) =- \tau [m]$$
for all $m\in \ker ( \bar\pi(1-\bar\sigma^{-1}))$ if $i$ is odd and for all $m\in \ker ({\overline{P'(\pi)}} \bar \Tr)$ if $i>0$ is even.  Thus, it induces an isomorphism between $\HH_i(A/(p)) \subset E^2_{i, 0}$ and
$\tau \HH_{i-2}(A/(p)) \subset E^2_{i-2, 1}$ if $i>2$, and an inclusion when $i=2$.
\end{prop}

We will postpone the proof to the end of the section, and see how this Proposition lets us complete the proof of Theorem \ref{onemodponenot}.
Note that the spectral sequence (\ref{Brun}) is multiplicative with respect to multiplication by the corresponding spectral sequence for $\THH(\Z;\Z/(p))$, since $\Z$ lies in the center of $A$.  There we have $\F_p[\tau]/\tau^2 \otimes  \F_p[u]$ with $d^2(u)=\tau$ mapping into the $u$ and $\tau$ that we have in (\ref{Brun}).  Thus for $[m]\in \HH_i(A/(p))$ and $j\geq 0$,
$$d^2(u^j[m]) =\tau j u^{j-1}[m]+ u^j d^2([m]) = \tau (j u^{j-1}[m]- u^j [m]).$$
When $i=2k$ is even, we get that $d^2:E^2_{i, 0}\to E^2_{i-2, 1}$ is given by
\begin{align}\label{dtwo}
& d^2  (\sum_{i=0}^ k u^{k-i} [a_i])
= \tau (
u^{k-1} [ka_0-a_1] +
\!\cdots \!+
u[2a_{k-2}-a_{k-1}] + 1[a_{k-1}-a_k])
\end{align}
for all $a_0\in T/(p)$, $a_i\in \ker ( {\overline{P'(\pi)}} \bar \Tr)$ for $1\leq i\leq k$.  

Equation (\ref{dtwo}) means that $\im( d^2)$ consists of all those  $\tau \sum_{i=0}^ {k-1}  u^{k-1-i} [b_i]$ where 
$$[b_0] \in  (kT/(p)+ \ker ( {\overline{P'(\pi)}} \bar \Tr) )/ \im ( \bar\pi(1-\bar\sigma^{-1}))$$
 and 
 $$b_i\in \ker ( {\overline{P'(\pi)}} \bar \Tr)$$ 
 for $1\leq i\leq k-1$.  Thus, $d^2$ surjects onto $E^2_{2k-2, 1}$ when $p \nmid k$ and so $kT/(p)=T/(p)$.  
 
 If $p|k$, the image of $d^2$ consists of all those  $\tau \sum_{i=0}^ {k-1}  u^{k-1-i} [b_i]$ where 
$[b_0] \in  \ker ( {\overline{P'(\pi)}} \bar \Tr) )/ \im ( \bar\pi(1-\bar\sigma^{-1}))$, but there are no restrictions on the $b_i \in \ker ( {\overline{P'(\pi)}} \bar \Tr)$  for $1\leq i\leq k-1$.  So if $p|k$, the cokernel of $d^2:E^2_{i, 0}\to E^2_{i-2, 1}$
is isomorphic to the quotient of $(T/(p)) /\im ( \bar\pi(1-\bar\sigma^{-1}))$ by $\ker ( {\overline{P'(\pi)}} \bar \Tr) )/ \im ( \bar\pi(1-\bar\sigma^{-1})$, which we now analyze.
We showed that $\Tr: T\to S$ is surjective, yielding a short exact sequence of $S$-modules $0\to\ker(\Tr)\to T\to S\to 0$ which must split.  The splitting is not the obvious inclusion $S\to T$.  If $p\nmid n$, we  can take $1/n$ times the inclusion as our splitting, but we cannot do this if $n$ is not a unit. The splitting does in any case give  a splitting of $S/(p)$-modules $T/(p)\cong \ker(\Tr)/(p) \oplus S/(p)$.   We intend to divide by $\ker ( {\overline{P'(\pi)}} \bar \Tr) )$ which contains  the first summand, so we can ignore the first summand.  The first summand  also  contains all of $ \im ( \bar\pi(1-\bar\sigma^{-1}))$.   So our cokernel is the quotient of the second summand $S/(p)$ by its intersection with $\ker ( {\overline{P'(\pi)}} \bar \Tr) $.  
By construction, $\bar\Tr$ sends this second summand isomorphically onto $S/(p)$.  Thus the only way an element of it can be in $\ker ( {\overline{P'(\pi)}} \bar \Tr)$ is if it is in $\Ann_{S/(p)}  {\overline{P'(\pi)}}$.  We get that the cokernel is 
$$S/(p)/(  \Ann_{S/(p)}  {\overline{P'(\pi)}} )\cong 
\begin{cases}
( \F_S[\pi]/(\pi^d)) / ( \F_S[\pi]/(\pi^d))\cong 0 &\mathrm{if}\ p|d,\\
( \F_S[\pi]/(\pi^d)) / (\pi \F_S[\pi]/(\pi^d))\cong \F_S  &\mathrm{if }\ p\nmid d.
\end{cases}
$$

Equation (\ref{dtwo}) also means that 
$\ker(d^2)$ consists of all those $ \sum_{i=0}^ k u^{k-i} [a_i]$ where $[ka_0] \in  \ker ( {\overline{P'(\pi)}} \bar \Tr)/ \im ( \bar\pi(1-\bar\sigma^{-1}))$ and the other $[a_i]$ are calculated inductively from  $[a_0]$ by
the requirement that $[ia_{k-i}-a_{k-(i-1)}] = [0]$ for all $1\leq i\leq k$.
Thus, if $p$ does not divide $k$, the kernel is exactly isomorphic to 
$$ \ker ( {\overline{P'(\pi)}} \bar \Tr)/ \im ( \bar\pi(1-\bar\sigma^{-1}))\cong \HH_2(A/(p))\cong \HH_{2a}(A/(p))\mathrm{\  for\  any\ } a>0,$$
 whereas if $p|k$, $[ka_0]=0$ and so the kernel is is all of $\HH_0(A/(p))$.

The calculation for $i$ odd is similar, but easier since all the odd Hochschild homology groups of $A/(p)$ are isomorphic.  So when $i$ is odd, $d^2:E^2_{i, 0}\to E^2_{i-2, 1}$ is always surjective, with kernel always isomorphic to $\HH_1(A/(p))$.

We gather all this information together to get 
\begin{equation*}
E^3_{s,t} \cong E^\infty_{s,t}\cong 
\begin{cases}
\HH_0(A/(p)) & t=0,\ s=2k,\ p|k\\ 
\HH_{2k}(A/(p))  & t=0,\ s=2k,\ p\nmid k\\ 
\HH_{2k-1}(A/(p))  & t=0,\ s=2k-1\\ 
\F_S  & t=1,\ s=2k-2,\ p|k\ \mathrm{but}\ p\nmid d\\
0  &  \mathrm{otherwise} \end{cases}
\end{equation*}
If we look only at the structure of vector spaces over $\F_S$, there can be no nontrivial extensions.  Even if we want to get the full $S$-module structure, that is: the $S/(p)\cong \F_S[\pi]/(\pi^d)$-module structure, the only case in which we have more than one nonzero module on a diagonal is in dimensions $2k-1$ if $p\nmid d$ but $p|k$.  The final result is an isomorphism of $S$-modules
\begin{align*}
\THH_i(A, & A/(p))\cong 
\\&
\begin{cases}
\HH_0(A/(p)) \cong  \F_S[\pi]/(\pi^d) \oplus \F_S^ {\oplus (n-1)}& i=2k,\ p|k\\ 
\HH_{2k}(A/(p)) \cong  \F_S[\pi]/(\pi^d, d\pi^{d-1} ) \oplus \F_S^ {\oplus (n-1)}&i=2k,\ p\nmid k\\ 
  \F_S[\pi]/(\pi^d) \oplus \F_S^ {\oplus (n-1)}& i=2k-1,  \ p|k\\
\HH_{2k-1}(A/(p))\cong  \F_S[\pi]/(\pi^d, d\pi^{d-1} ) \oplus \F_S^ {\oplus (n-1)}  & i=2k-1,\ p\nmid k.\\ 
\end{cases}
\end{align*}
The only case requiring justification is the extension
$$0\to E^\infty_{2k-2, 1} \to \THH_{2k-1} (A, A/(p)) \to E^\infty_{2k-1,0}\to 0$$
 which takes the form
$$0\to\F_S\to  \THH_{2k-1} (A, A/(p)) \to \F_S[\pi]/(\pi^{d-1} ) \oplus \F_S^ {\oplus (n-1)} \to 0$$
when $p\nmid d$ but $p|k$.  In that case,  we claim that the extension is nontrivial.  

In the case where $p\nmid n$  this follows directly from of \cite[Proposition 5.6({\it ii})]{LM00}, where the analogous extension problem is solved for $S$.   By Corollary \ref{tamedivisionalgebamodp}, the first summand in $\HH_i(A/(p) )\cong \HH_i(S/(p)) \oplus \F_S^ {\oplus (n-1)}$ for all $i$ is the isomorphic image of $\HH_i(S/(p))$ by the map induced by the inclusion $S\hookrightarrow A$.  By the naturality of the Brun spectral sequence (\ref{Brun}), the extension problem we see here is exactly that in \cite{LM00}, where the result is that $\THH_{2k-1}(S, S/(p)) \cong  \F_S[\pi]/(\pi^d)$.  The only difference is that in our calculation, we also have an extra copy of 
$ \F_S^ {\oplus (n-1)}$ added on as a direct summand in the whole group and in the quotient.

In the case where $p\mid n$, the analysis of the extension in \cite[Proposition 5.6({\it ii})]{LM00} is completely analogous for $T$ and for $S$, so if we are only interested in the $S$-module structure, the extension problem in the analogous calculation of $ \THH_{2k-1} (T, T/(p))$ using  the Brun spectral sequence is 
$$0\to\F_S^ {\oplus (n)} \to  \THH_{2k-1} (T, T/(p)) \to( \F_S[\pi]/(\pi^{d-1} ) ) ^ {\oplus (n)} \to 0.$$
 The conclusion, again if we are only interested in the $S$-module structure, is that $\THH_{2k-1} (T, T/(p)) \cong ( \F_S[\pi]/(\pi^{d} ) ) ^ {\oplus (n)}$.  Instead of Corollary \ref{tamedivisionalgebamodp}, we have the weaker Corollary \ref{wilddivisionalgebamodp}.  We still have the naturality of the Brun spectral sequence.  Its spectral sequence differentials respect the decomposition of $\HH_*(A/(p))$ and $\THH_*(A/(p))$ because of the explicit formula in Proposition \ref{keylemma}.  So the inclusion $T\hookrightarrow A$ induces maps sending the Brun spectral sequence $E^\infty_{2k-2, 1}$  for  
$\THH_{2k-1} (T, T/(p))$ surjectively onto the
$E^\infty_{2k-2, 1}$  for  $\THH_{2k-1} (A, A/(p))$, and sending the  $E^\infty_{2k-1, 0}$  for  $\THH_{2k-1} (T, T/(p))$ surjectively onto the first factor in the $E^\infty_{2k-1, 0} \cong \F_S[\pi]/(\pi^{d-1} ) \oplus \F_S^ {\oplus (n-1)}$ we have in the calculation of $\THH_{2k-1} (A, A/(p))$.

The nontriviality of the extension for $A$ can therefore be deduced from the nontriviality of the analogous extension for $T$ by the following algebraic lemma, for $\F=\F_S$ and $M$ the image of $\THH_{2k-1} (T, T/(p))$  in $\THH_{2k-1} (A, A/(p))$.
\end{proof}

\begin{lem}\label{extensionimage}
Let $\F$ be a field and let $d\geq 2$, $n\geq 1$ be integers.  Assume that in the commutative diagram of $\F[x]/(x^d)$-modules
$$\xymatrix{
0\ar[r]
 &{( x^{d-1} \F[x]/(x^d) )^ {\oplus n} } \ar[r] ^{i}\ar[d] ^{f_0}
 & {( \F[x]/(x^d) )^ {\oplus n} }\ar[r]^q \ar[d] ^{f_1}
& {( \F[x]/(x^{d-1}) )^ {\oplus n} } \ar[r] \ar[d]^ {f_2}
& 0\\
0\ar[r]
& \F \ar[r]^{i_M}
& M  \ar[r]
& {\F[x]/(x^{d-1}) } \ar[r]
& 0
}
$$
the maps $i$ and $q$ in the top row are the usual inclusion and quotient maps, the bottom row is also exact, and the maps $f_0$ and $f_2$ are surjections.  Then we have an $\F[x]/(x^d)$-module isomorphism
$M\cong \F[x]/(x^d)$.
\end{lem}

\begin{proof} (of Lemma \ref{extensionimage})

Since $f_0$ is surjective, there exist some elements $\gamma_1, \gamma_2,\ldots, \gamma_n\in \F$ so that $i_M(1_\F)= f_1 (\gamma_1 x^{d-1}, \gamma_2 x^{d-1}, \ldots, \gamma_nx^{d-1})$.  The elements $\gamma_1, \gamma_2,\ldots, \gamma_n$ cannot all be zero because $i_M$ is an injection.  Set $\underline{v}=(\gamma_1, \gamma_2,\ldots, \gamma_n)$.  Then $f_1(\underline{v})$ generates over 
$\F[x]/(x^d)$ a free submodule of $M$.  This is because a cyclic $\F[x]/(x^d)$-module is either free or it is a $\F[x]/(x^{d-1})$-module, but we know that
$$x^{d-1}f_1(\underline{v})=f_1(x^{d-1} \underline{v}) =i_M(1_\F)\neq 0.$$
Once we know that $\F[x]/(x^d)\cdot f_1(\underline{v})$ is a free $\F[x]/(x^d)$-module, we know that it has dimension $d=\dim_\F M$ over $\F$, so by counting dimensions it must be equal to all of $M$.
\end{proof}
\begin{proof} (of Proposition \ref{keylemma})

For any abelian group $A$ and a pointed simplicial set $X.$, let $A(X.)=A[X.]/A\cdot *$.  If $A$ is a simplicial abelian group, we define $A(X.)$ similarly, taking the diagonal of the resulting bisimplicial abelian group.  For any rings $R$ and  $A$, we let $R(A)=R[A]/R\cdot 0$ be the the ring which is additively a free $R$-module on $A$'s nonzero elements, with the inherited multiplication; for any left $R$-algebra $M$ and right $R$-algebra $N$, let $B.(M,R,N)= M\blb R\blb R\cdots  R \blb  N\brb\cdots\brb\brb$ be the bar construction using iterations of the previous construction, with $k$ copies of $R$ in degree $k$.  Since $B.(R,R,N)$ is a free $R$-resolution of $N$, $\pi_*(B.(M,R,N))=\Tor_*^R(M,N)$.  There is an obvious left $M$-module structure on  $B.(M,R,N)$, and also a right $\mathbb{Z}(N)$-module structure where the $\mathbb{Z}$  multiplies into the $M$ on the left. 

For an $R$-bimodule $M$, let $V.(R,M)$ be the Hochschild complex of $R$ with coefficients in $M$, $V_k(R,M)=M\otimes R^{\otimes k}$.  Let ${ \bf V}.(R, M)$ be the stabilization of $V.(\mathbb{Z}(R),M)$ given by
$$\mathrm{holim}_{I^{k+1}}\mathrm{s.Ab}(\mathbb{Z}(S^{x_0}.) \otimes\cdots\otimes \mathbb{Z}(S^{x_k}. ),
M(S^{x_0}.)\otimes \mathbb{Z}(R) (S^{x_1}.)\otimes\cdots\otimes  \mathbb{Z}(R) (S^{x_k}.)),$$
the mapping space in simplicial abelian groups, where $x_i\in I$ and $I$ is the skeleton of the category of finite sets and injective maps.
Consider the map 
$$V.(\ZAp, B.( \Ap ,  A ,  \Ap )) \longrightarrow^{\!\!\!\!\!\!\!\!\!\!\!\Sigma_V} { \bf V}.( \Ap , B.( \Ap ,  A , \Ap )) $$
from \cite[Equation (3.11)]{LM00}, which comes from including  the linear case of $x_i=\emptyset$ for every $i$ into the limit.  As explained above  \cite[Equation (3.7)]{LM00},
 ${ \bf V}.(R; M)$  is a model for $\THH(R;M)$.  By a comparison to a double bar construction from Lemma 3.2 there,  \cite{LM00} also shows that  $ { \bf V}.( \Ap , B.( \Ap ,  A , \Ap ))$ is a model for $\THH( A ,  \Ap )$.  
Filtering by the  simplicial degree in $V$, $\Sigma_V$ induces a map of  $E^2$ spectral sequences
$$ \HH_*(\ZAp, \Tor_*^ A  ( \Ap ,  \Ap ))  \longrightarrow^{\!\!\!\!\!\!\!\!\!\!\!\Sigma_V} 
\THH_*( \Ap , \Tor_*^ A  ( \Ap ,  \Ap ))$$ 
where the target is exactly the Brun spectral sequence we are working with. To calculate $d^2$ on elements in $E^2_{i,0}$ in the Brun spectral sequence, we will find elements which map to them via  $\Sigma_V$ and calculate $d^2$ there.  When we look at $V_r (\ZAp, B_s( \Ap ,  A ,  \Ap ))$ as a double complex, we will call its horizontal (in the $r$-direction) differential $d_{\HH}=\sum_{i=0}^n (-1)^i d_{i,\HH}$ and its vertical (in the $s$-direction) differential $d_{\mathrm{bar}}$.

When Larsen showed that $\HH_*(A)$ can be calculated using the small complex (\ref{small}),
he used maps $i_*$  (see the proof of  \cite[Proposition 3.8]{L95}) to map this complex into the standard Hochschild complex, and then showed that the reduced Hochschild complex splits as a direct sum of the image of the $i_*$ and an acyclic summand.
These maps are given for any $m\in T$ by
\begin{align*}
& i  _{2k} ( m)  
 = 
\sum_{i_1, i_2,\ldots , i_k=1}^{nd}  
m x ^{i_1 + \cdots + i_k - k }
\otimes x \otimes  x^{nd-i_1}\otimes \cdots
\otimes  x \otimes x^{nd-i_k}, \\
& i  _{2k+1} (m)  = 
\sum_{i_1, i_2,\ldots , i_k=1}^{nd}  
m x ^{n+i_1 + \cdots + i_k -(k+1) }
\otimes x \otimes  x^{nd-i_1}\otimes \cdots
\otimes  x \otimes  x^{nd-i_k}\otimes  x.
\end{align*}
We want to use Larsen's $i_*$ in order to identify generators of $\HH_*(\Ap)$.  The reduction mod $p$ of his $i_*$ gives a chain of quasi-isomorphisms from the reduction mod $p$ of the small complex above to the reduction mod $p$ of the Hochschild complex.  This is because as explained above, the reduced Hochschild complex is a direct sum of the isomorphic image of $i_*$ and an acyclic complex, and the map from the standard Hochschild complex to the reduced Hochschild complex is a quasi-isomorphism.  After reducing mod $p$, the acyclic complex remains acyclic, and the map between the standard and the reduced Hochschild complexes remains a quasi-isomorphism.  But rather than mapping  into 
$ A ^ {\otimes (k+1)}$, as his $i_k$ does, we will define maps
$$\tilde i_k:\ T\to  \Ap \blb  \Ap  \brb \otimes \mathbb{Z}( \Ap )^{\otimes k}\in V_k(\Ap(\Ap), B_0( \Ap ,  A ,  \Ap ))$$
that reduce (using the 
maps $\mathbb{Z}( \Ap )\to \Ap $ and $ \Ap \blb  \Ap \brb \to  \Ap $)
 to the same elements that Larsen's $i_k$ do in $\bigl( \Ap \bigr) ^ {\otimes (k+1)}$ and therefore can be used to represent all of $\HH_*( \Ap )$.  
 Looking at the reduction modulo $p$ of Larsen's small complex (\ref{small}), we see that to
 represent all  the elements in  $\HH_i( \Ap )$, $i\geq 0$, it will be enough to look at $i_0(m)$ for all $m\in T$, $i_{2k}(m)$ for all $m$ such that $\bar m\in \ker ({\overline{P'(\pi)}} \bar \Tr)$ for all $k>0$, and $i_{2k+1}(m)$ for all $m$ such that $\bar m\in \ker ( \bar\pi(1-\bar\sigma^{-1}))$ for all $k\geq 0$.

%begin imported material

We start with the odd-dimensional case.  For $k\geq 1$, let
\begin{align*}
\tilde i & _{2k+1} (m)  \\
& = 
\sum_{i_1, i_2, \ldots, i_k=1}^{nd} \bar m \left( \bar x ^{n+i_1+i_2+\cdots + i_k -\left(k+1\right) }\right)
\otimes \left(\bar x\right) \otimes \left(\bar x^{nd-i_1}\right)\otimes \cdots
 \otimes \left(\bar x^{nd-i_k}\right)\otimes \left(\bar x\right).
\end{align*}
Then  $(d_{0,HH}-d_{1,HH})(\tilde i_{2k+1}(m))=0$ because the $i_1=a$ term in $d_{0,HH}$ is cancelled by the  $i_1=a+1$ in $d_{1,HH}$, except when $i_1=1$ and when $i_1=nd$.  When  $i_1=nd$,  $ \bar x ^{n+i_1+i_2+\cdots + i_k -\left(k+1\right)}  =0$  because $n\geq 1$ and so $n+nd+i_2+\cdots + i_k -\left(k+1\right) + 1\geq  nd$, and $\bar x^{nd}=0$.  When $i_1=1$,  $\bar x^{1+nd-1}=0$ . 

We also get that $(d_{2,HH}-d_{3,HH})(\tilde i_{2k+1}(m))=0$, since
\begin{align*}
&(d_{2,HH}-d_{3,HH})(\tilde i_{2k+1}(m))  \\
& = \sum_{i_1, i_2, \ldots, i_k=1}^{nd} \bar m \left(\bar x^{n+i_1+i_2+\cdots + i_k-(k+1)}\right)\otimes \left(\bar x\right)\otimes \left(\bar x^{nd-i_1+1}\right) \otimes \left(\bar x^{nd-i_2}\right)\otimes \cdots \otimes \left(\bar x\right)\\
& -  \sum_{i_1, i_2, \ldots, i_k=1}^{nd} \bar m \left(\bar x^{n+i_1+i_2+\cdots + i_k-(k+1)}\right)\otimes \left(\bar x\right)\otimes \left(\bar x^{nd-i_1}\right) \otimes \left(\bar x^{nd-i_2+1}\right)\otimes \cdots \otimes \left(\bar x\right)
\end{align*}
and the $(i_1,i_2)=(a,b)$ in the first sum cancels the $(i_1,i_2)=(a-1,b+1)$ in the second sum. The terms which are left are when $i_1=1$ or $i_2=nd$ in the first sum, and when $i_1=nd$ or $i_2=1$ in the second sum. These terms are zero because $\bar x^{nd}=0$: if $i_1=1$ for the first sum or $i_2=1$  in the second, this is obvious.  If $i_2=nd$ in the first sum, then $n+i_1+i_2+\cdots + i_k-(k+1) \geq nd$ if $i_1\geq 2$, but if $i_1=1$ then we already know we get zero in the first tensor coordinate, and similarly for $i_1=nd$ in the second sum. 

 By the same argument, 
$$(d_{4,HH}-d_{5,HH})(\tilde i_{2k+1}(m)) = \cdots = (d_{2k-2,HH}-d_{2k-1,HH})(\tilde i_{2k+1}(m))=0.$$
So
\begin{align*}
&d_{HH}\left(\tilde i_{2k+1}(m)\right) = d_{2k,HH}(\tilde i_{2k+1}(m))-d_{2k+1,HH}(\tilde i_{2k+1}(m))\\
& = \sum_{i_1, i_2, \ldots, i_k=1}^{nd} \bar m \left( \bar x^{n+i_1+i_2+\cdots + i_k-(k+1)}\right)\otimes \left(\bar x\right)\otimes \cdots \otimes \left(\bar x\right)\otimes \left(\bar x^{nd-i_k+1}\right)\\
& -  \sum_{i_1, i_2, \ldots, i_k=1}^{nd} \bar x\bar m \left( \bar x^{n+i_1+i_2+\cdots + i_k-(k+1)}\right)\otimes \left(\bar x\right)\otimes \cdots \otimes \left(\bar x\right)\otimes \left(\bar x^{nd-i_k}\right)\\
& = \sum_{i_1, i_2, \ldots, i_k=1}^{nd} \bar m \left( \bar x^{n+i_1+i_2+\cdots + i_k-(k+1)}\right)\otimes \left(\bar x\right)\otimes  \cdots \otimes \left(\bar x\right)\otimes \left(\bar x^{nd-i_k+1}\right)\\
& - \!\!\! \sum_{i_1, i_2, \ldots, i_k=1}^{nd} \!\!\!\sigma^{-1}\left(\bar m\right)\bar x \left( \bar x^{n+i_1+i_2+\cdots + i_k-(k+1)}\right)\otimes \left(\bar x\right)\otimes \cdots \otimes \left(\bar x\right)\otimes \left(\bar x^{nd-i_k}\right)\\
& =\!\!\! \sum_{i_1, i_2, \ldots, i_{k-1}=1}^{nd}\sum_{i_k=1}^{nd} \left(\bar m\left( \bar x^{n+i_1+i_2+\cdots + i_k-k}\right)-  \sigma^{-1}\left(\bar m\right)\bar x \left( \bar x^{n+i_1+i_2+\cdots + i_k-(k+1)}\right)\right) \otimes
\\& \qquad\qquad\qquad\qquad\qquad\qquad\qquad\qquad
 \left(\bar x\right)\otimes  \cdots \otimes \left(\bar x\right)\otimes \left(\bar x^{nd-i_k}\right)\\
& -  \sum_{i_1, i_2, \ldots, i_{k-1}=1}^{nd} \bar m \left( \cancel{\bar x^{n+nd + i_1+i_2+\cdots i_{k-1}-k}}\right)\otimes \left(\bar x\right)\otimes  \cdots \otimes \left(\bar x\right)\otimes \left(1\right)\\
& +\sum_{i_1, i_2, \ldots, i_{k-1}=1}^{nd}  \bar m \left( \bar x^{n+i_1+i_2+\cdots + i_{k-1}-(k+1)}\right)\otimes \left(\bar x\right)\otimes  \cdots \otimes \left(\bar x\right)\otimes \left(\cancel{\bar x^{nd}}\right)\\
& =\!\!\!  \sum_{i_1, i_2, \ldots, i_{k-1}=1}^{nd}\sum_{i_k=1}^{nd} \left(\bar m\left( \bar x^{n+i_1+i_2+\cdots + i_k-k}\right)-  \sigma^{-1}\left(\bar m\right)\bar x \left( \bar x^{n+i_1+i_2+\cdots + i_k-(k+1)}\right)\right) \otimes \\& \qquad\qquad\qquad\qquad\qquad\qquad\qquad\qquad
\left(\bar x\right)\otimes \cdots \otimes \left(\bar x\right)\otimes \left(\bar x^{nd-i_k}\right)
\\
& = d_{\mathrm{bar}}\Bigl(\!\sum_{i_1, i_2, \ldots, i_{k-1}=1}^{nd}\sum_{i_k=1}^{nd} 
 \\& \qquad\qquad\qquad
 \left( -\bar m \left(x^n\left(\bar x^{i_1+i_2+\cdots+i_k-k}\right)\right)  
+ \sigma^{-1}\left( \bar m\right)\bar x\left(x^{n-1}\left(\bar x^{i_1+i_2+\cdots+i_k-k}\right)\right)\right)
\\& \qquad\qquad\qquad\qquad\qquad\qquad\qquad\qquad
\otimes\left(\bar x\right)\otimes \cdots \otimes \left(\bar x\right)\otimes \left(\bar x^{nd-i_k}\right)\Bigr)
\end{align*}
because $\left(\bar m - \sigma^{-1}\left(\bar m\right)\right)\bar x^n = \left(\bar m - \sigma^{-1}\left(\bar m\right)\right)\bar \pi = 0$.
This tells us that in our  spectral sequence, $d^2\left(\tilde i _{2k+1} \left( m\right) \right)$ is equal to the
class in $E^2$ of 
\begin{align*}
&d_{HH}\Bigl(\!\sum_{i_1, i_2, \ldots, i_{k-1}=1}^{nd}\sum_{i_k=1}^{nd} 
 \\& \qquad\qquad\qquad
 \left( -\bar m \left(x^n\left(\bar x^{i_1+i_2+\cdots+i_k-k}\right)\right)  
+ \sigma^{-1}\left( \bar m\right)\bar x\left(x^{n-1}\left(\bar x^{i_1+i_2+\cdots+i_k-k}\right)\right)\right)
\\& \qquad\qquad\qquad\qquad\qquad\qquad\qquad\qquad
\otimes \left(\bar x\right)\otimes 
 \left(\bar x^{nd-i_1}\right)\otimes \cdots \otimes \left(\bar x\right)\otimes \left(\bar x^{nd-i_k}\right)\Bigr)
\end{align*}

Note that on this element, $$d_{0,HH}-d_{1,HH}=d_{2,HH}-d_{3,HH}= \cdots= d_{2k-2,HH}-d_{2k-1,HH}=0$$ as before, leaving only $d_{2k, HH}$.  So 
\begin{align*}
 &d^2\left(\tilde i _{2k+1} \left( m\right) \right) \\
 &= \Bigl[\sum_{i_1, i_2, \ldots, i_{k-1}=1}^{nd}\sum_{i_k=1}^{nd} 
  \\& 
  \bigl(-\bar x^{nd-i_k}\bar m \left(x^n\left(\bar x^{i_1+i_2+\cdots+i_k-k}\right)\right)
    +\bar x^{nd-i_k}\sigma^{-1}\left(\bar m\right)\bar x\left(x^{n-1}\left(\bar x^{i_1+i_2+\cdots+i_k-k}\right)\right)\bigr)\\
  & \qquad  \qquad \qquad \qquad \qquad \qquad\qquad \qquad \qquad \qquad
  \otimes \left(\bar x\right)
  \otimes \left(\bar 
 x^{nd-i_1}\right)\otimes \cdots \otimes \left(\bar x\right)\Bigr]
\end{align*}
Analyzing the $0$'th coordinate of this for fixed $i_1, i_2,\ldots, i_{k-1}$, we get
\begin{align*}
&\sum_{i_k=1}^{nd} \!\!\left(-\bar x^{nd-i_k}\bar m \left(x^n\left(\bar x^{i_1+i_2+\cdots+i_k-k}\right)\right)+\bar x^{nd-i_k}\sigma^{-1}\left(\bar m\right)\bar x\left(x^{n-1}\left(\bar x^{i_1+i_2+\cdots+i_k-k}\right)\right)\right)\\
&=\sum_{i_k=1}^{nd} \left(-\bar x^{nd-i_k}\bar m \left(x^n\left(\bar x^{i_1+i_2+\cdots+i_k-k}\right)\right)+\bar x^{nd-i_k+1}\bar m \left(x^{n-1}\left(\bar x^{i_1+i_2+\cdots+i_k-k}\right)\right)\right)\\
& = \sum_{i_k=1}^{nd} \left(-\bar x^{nd-i_k}\bar m \left(x^n\left(\bar x^{i_1+i_2+\cdots+i_k-k}\right)\right)+\bar x^{nd-i_k}\bar m \left(x^{n-1}\left(\bar x^{i_1+i_2+\cdots+i_k+1-k}\right)\right)\right)\\
&- \bar m \left(x^{n-1}\left(\bar x^{nd+i_1+i_2+\cdots+i_{k-1}+1-k}\right)\right) + \bar x^{nd}\bar m\left(x^{n-1}\left(\bar x^{i_1+i_2+\cdots +i_{k-1}+1-k}\right)\right) \\
& = \sum_{i_k=1}^{nd} \left(-\bar x^{nd-i_k}\bar m \left(x^n\left(\bar x^{i_1+i_2+\cdots+i_k-k}\right)\right)+\bar x^{nd-i_k}\bar m \left(x^{n-1}\left(\bar x^{i_1+i_2+\cdots+i_k+1-k}\right)\right)\right).
\end{align*}
Since
\begin{align*}
& d_{\mathrm{bar}}\left(\bar x^{nd-i_k}\bar m\left(x^{n-1}\left(x\left(\bar x^{i_1+i_2+\cdots+i_k-k}\right)\right)\right)\right)\\
 & \qquad = \bar x^{nd-i_k}\bar m\bar x^{n-1}\left( x\left(\bar x^{i_1+i_2+\cdots+i_{k}-k}\right)\right)
 - \bar x^{nd-i_k}\bar m\left( x^n\left(\bar x^{i_1+i_2+\cdots+i_{k}-k}\right)\right)\\
 &\qquad\qquad\qquad\qquad\qquad\qquad\qquad\qquad
 + \bar x^{nd-i_k}\bar m\left( x^{n-1}\left(\bar x^{i_1+i_2+\cdots+i_{k}+1-k}\right)\right),
\end{align*}
this is homologous to $ \sum_{i_k=1}^{nd}-\bar x^{nd-i_k}\bar m\bar x^{n-1}\left(x\left(\bar x^{i_1+i_2+\cdots +i_k-k}\right)\right)$, and since
\begin{align*}
& d_{\mathrm{bar}}\left(\bar x^{nd-i_k}\bar m\bar x^{n-1}\left(x\left(x^{i_k-1}\left(\bar x^{i_1+i_2+\cdots +i_{k-1}-\left(k-1\right)}\right)\right)\right)\right)\\
&\qquad = \bar x^{nd-i_k}\bar m\bar x^{n}\left(x^{i_k-1}\left(\bar x^{i_1+i_2+\cdots +i_{k-1}-\left(k-1\right)}\right)\right)\\
&\qquad
 - \bar x^{nd-i_k}\bar m\bar x^{n-1}\left(x^{i_k}\left(\bar x^{i_1+i_2+\cdots +i_{k-1}-\left(k-1\right)}\right)\right)\\
&\qquad
 + \bar x^{nd-i_k}\bar m\bar x^{n-1}\left(x\left(\bar x^{i_1+i_2+\cdots +i_{k-1}+i_k-k}\right)\right),
\end{align*}
that is homologous to
\begin{align*}
&\sum_{i_k=1}^{nd} \Bigl(\bar x^{nd-i_k}\bar m\bar x^{n}\left(x^{i_k-1}\left(\bar x^{i_1+i_2+\cdots +i_{k-1}-\left(k-1\right)}\right)\right) \\
&\qquad\qquad-\bar x^{nd-i_k}\bar m\bar x^{n-1}\left(x^{i_k}\left(\bar x^{i_1+i_2+\cdots +i_{k-1}-\left(k-1\right)}\right)\right)\Bigr)\\
& ={\bar x^{nd-1}\bar m\bar x^{n}}\left(1\left(\bar x^{i_1+i_2+\cdots+i_{k-1}-\left(k-1\right)}\right)\right) - \bar m \bar x^{n-1} \left(x^{nd}\left(\bar x^{i_1+i_2+\cdots+ i_{k-1}-\left(k-1\right)}\right)\right)\\
&\qquad\qquad + \sum_{i_k=2}^{nd} \Bigl(\bar x^{nd-i_k}\bar m\bar x^{n}\left(x^{i_k-1}\left(\bar x^{i_1+i_2+\cdots +i_{k-1}-\left(k-1\right)}\right)\right) \\
&\qquad\qquad\qquad -\bar x^{nd- i_k}\sigma^{-1}( \bar m)\bar x^{n}\left(x^{i_k-1}\left(\bar x^{i_1+i_2+\cdots +i_{k-1}-\left(k-1\right)}\right)\right)\Bigr)\\
& = -\bar m \bar x^{n-1} \left(x^{nd}\left(\bar x^{i_1+i_2+\cdots+ i_{k-1}-\left(k-1\right)}\right)\right),
\end{align*}
because $\bar x^{nd}=0$ and because $m$ was chosen to satisfy $\left(\bar m - \sigma^{-1}\left(\bar m\right)\right)\bar x^n=0$.   This tells us what $d^2\left(\tilde i _{2k+1} \left( m\right) \right) $ is; however, to get it into a more familiar form we
observe that
\begin{align*}
& d_{\mathrm{bar}}\left(\bar m\left(x^{n-1}\left(x^{nd}\left(\bar x^{i_1+i_2+\cdots+i_{k-1}-\left(k-1\right)}\right)\right)\right)\right)\\
& = \bar m \bar x^{n-1} \left(x^{nd}\left(\bar x^{i_1+i_2+\cdots+ i_{k-1}-\left(k-1\right)}\right)\right) - \bar m \left(x^{nd+n-1}\left(\bar x^{i_1+i_2+\cdots+ i_{k-1}-\left(k-1\right)}\right)\right)\\
& + \bar m \left(x^{n-1}\left(\cancel{\bar x^{nd+i_1+i_2+\cdots+i_{k-1}-\left(k-1\right)}}\right)\right)\\
& = \bar m \bar x^{n-1} \left(x^{nd}\left(\bar x^{i_1+i_2+\cdots+ i_{k-1}-\left(k-1\right)}\right)\right) - \bar m \left(x^{nd+n-1}\left(\bar x^{i_1+i_2+\cdots+ i_{k-1}-\left(k-1\right)}\right)\right)
\end{align*}
and
\begin{align*}
& d_{\mathrm{bar}}\left(\bar m\left(x^{nd}\left(x^{n-1}\left(\bar x^{i_1+i_2+\cdots+i_{k-1}-\left(k-1\right)}\right)\right)\right)\right)\\
& = \bar m \cancel{\bar x^{nd}} \left(x^{n-1}\left(\bar x^{i_1+i_2+\cdots+ i_{k-1}-\left(k-1\right)}\right)\right)\\
&  - \bar m \left(x^{nd+n-1}\left(\bar x^{i_1+i_2+\cdots+ i_{k-1}-\left(k-1\right)}\right)\right) + \bar m \left(x^{nd}\left(\bar x^{n+i_1+i_2+\cdots+i_{k-1}-k}\right)\right)\\
& = -\bar m \left(x^{nd+n-1}\left(\bar x^{i_1+i_2+\cdots+ i_{k-1}-\left(k-1\right)}\right)\right) + \bar m \left(x^{nd}\left(\bar x^{n+i_1+i_2+\cdots+i_{k-1}-k}\right)\right).
\end{align*}
So 
\begin{align*}
 &d^2\left(\tilde i _{2k+1} \left( m\right) \right) \\
&\!=\!\! \left[\sum_{i_1, i_2, \ldots, i_{k-1}=1}^{nd}\sum_{i_k=1}^{nd} -\bar x^{nd-i_k}\bar m\bar x^{n-1}\left(x\left(\bar x^{i_1+i_2+\cdots+i_{k}-k}\right)\right)
\otimes \!\cdots \!\otimes \left(\bar x^{nd-i_{k-1}}\right)\otimes \left(\bar x\right)\right]\\
&\!=\!\! \left[\sum_{i_1, i_2, \ldots, i_{k-1}=1}^{nd}-\bar m\bar x^{n-1}\left(x^{nd}\left(\bar x^{i_1+i_2+\cdots+i_{k-1}-\left(k-1\right)}\right)\right)
 \otimes \cdots \otimes \left(\bar x^{nd-i_{k-1}}\right)\otimes \left(\bar x\right)\right]\\
&\!=\!\! \left[\sum_{i_1, i_2, \ldots, i_{k-1}=1}^{nd}- \bar m\left(x^{nd}\left(\bar x^{n+i_1+i_2+\cdots+i_{k-1}-k}\right)\right)\otimes \left(\bar x\right)
 \otimes \cdots \otimes \left(\bar x^{nd-i_{k-1}}\right)\otimes \left(\bar x\right)\right]
\end{align*}
Because of Lemma \ref{alsoneeded}, this says that that 
$$d^2(\tilde i _{2k+1} ( m) )= -\tau \cdot [ \overline{ i_{2k-1} ( m)}] \in \tau \cdot \HH_{2k-1} (\Ap ),$$
where we use $\overline{ i_{2k-1} ( m)}$ to denote the reduction mod $p$ of Larsen's $ i_{2k-1} (m)$.

\medskip
In the even case, we look at 
\begin{align*}
\tilde i & _{2k} ( m)  
& = 
\sum_{i_1, i_2, \ldots, i_k=1}^{nd} 
\bar m \left( \bar x ^{i_1+i_2+\cdots + i_k - k }\right)
\otimes \left(\bar x\right) \otimes \left(\bar x^{nd-i_1}\right)\otimes \cdots
\otimes \left(\bar x\right) \otimes \left(\bar x^{nd-i_k}\right) 
\end{align*}
for $m\in T$ for which $\bar m\in \ker ({\overline{P'(\pi)}} \bar \Tr)$.  
%Recall that ${\overline P'(\pi) }= d \bar x^{nd-n}$.
By an argument similar to the one we had in the odd case, $(d_{0,HH}-d_{1,HH})(\tilde i_{2k}(m))=0$.   Also by a similar argument to the odd case,
$$(d_{2,HH}-d_{3,HH})(\tilde i_{2k}(m))= \cdots = (d_{2k-2,HH}-d_{2k-1,HH})(\tilde i_{2k+1}(m))=0,$$
again using cancellations and the fact that $\bar x^{nd}=0$.  So
\begin{align*}  
& d_{\HH}(\tilde i  _{2k} ( m)  ) = d_{2k,\HH}(\tilde i  _{2k} ( m)  )\\
 & = 
\sum_{i_1, i_2, \ldots, i_k=1}^{nd}
\bar x^{nd-i_k}  \bar m \left(\bar x ^{i_1+i_2+\cdots + i_k - k }\right)
\otimes \left(\bar x\right) \otimes \left(\bar x^{nd-i_1}\right)\otimes \cdots
\otimes \left(\bar x\right) \\
& = 
\sum_{i_1, i_2, \ldots, i_k =1}^{nd}  
 \sigma^{i_k - nd} \left(\bar m\right)  \bar x^{nd-i_k}  \left( \bar x ^{i_1+i_2+\cdots + i_k - k }\right)
\otimes \left(\bar x\right) \otimes \left(\bar x^{nd-i_1}\right)  \otimes \cdots
\otimes \left(\bar x\right) \\
& = 
\sum_{i_1, i_2, \ldots, i_k =1}^{nd}  
 \sigma^{i_k } \left(\bar m\right)  \bar x^{nd-i_k}  \left( \bar x ^{i_1+i_2+\cdots + i_k - k }\right)
\otimes \left(\bar x\right) \otimes \left(\bar x^{nd-i_1}\right)  \otimes \cdots
\otimes \left(\bar x\right) .
\end{align*}

By the definition of $d _{\mathrm{bar}} $,
\begin{align*}
d & _{\mathrm{bar}}  \left(
\sum_{i_1, i_2, \ldots, i_k=1}^{nd} 
 \sigma^{i_k } \left(\bar m\right)  \left(  x^{nd-i_k}  \left( \bar x ^{i_1+i_2+\cdots + i_k - k }\right)\right)
\otimes \left(\bar x\right) \otimes \left(\bar x^{nd-i_1}\right)  \otimes \cdots\otimes \left(\bar x\right)   \right)\\
& = 
\sum_{i_1, i_2, \ldots, i_k=1}^{nd} 
 \sigma^{i_k } \left(\bar m\right)  \bar x^{nd-i_k}  \left( \bar x ^{i_1+i_2+\cdots + i_k - k }\right)
\otimes \left(\bar x\right) \otimes \left(\bar x^{nd-i_1}\right)  \otimes \cdots\otimes \left(\bar x\right) \\
&\  -
\!\!\!\sum_{i_1, i_2, \ldots, i_{k-1}=1}^{nd}  
\!\left( \sum_{i_k=1}^{nd} 
 \sigma^{i_k} \left(\bar m\right)   \left( \bar x ^ { nd + i_1+i_2+\cdots +  i_{k-1}  - k } \right)
\otimes \left(\bar x\right) \otimes \left(\bar x^{nd-i_1}\right)  \otimes \cdots \otimes \left(\bar x\right)
\! \right)  \\ 
 & = d_{\HH}(\tilde i  _{2k} ( m)  ) - \!\!\!\!\!\!\!\!\!\!\!\!\sum_{i_1, i_2, \ldots, i_{k-1}=1}^{nd}  
\!\!\!\!\!\!\!\!\!\!\!\! d\Tr\left(\bar m\right)   \left( \bar x ^ { nd + i_1+i_2+\cdots +  i_{k-1}  - k } \right) \!\otimes \!\left(\bar x\right)\!\otimes \!  \cdots \!\otimes \! \left(\bar x^{nd-i_{k-1}}\right) \!\otimes \! \left(\bar x\right)\\
 & = d_{\HH}(\tilde i  _{2k} ( m)  ) - 
d\Tr\left(\bar m\right)   \left( \bar x ^ { nd -1 } \right) \otimes \left(\bar x\right)  \cdots \otimes \left(\bar x^{nd-1}\right)
\otimes \left(\bar x\right)\\
& = d_{\HH}(\tilde i  _{2k} ( m)  ) +  d_{\mathrm{bar}}\left(d\Tr\left(\bar m\right)\left(x^{nd-n}\left(\bar x^{n-1}\right)\right)\otimes \left(\bar x\right)\otimes\cdots\otimes \left(\bar x^{nd-1}\right)\otimes\left(\bar x\right)\right)
 \end{align*}
where the last equality is because $d\Tr\left(\bar m\right) \bar x^{nd-n} = P'\left(\pi\right)\Tr\left(\bar m\right)=0$.
 So
\begin{align*}
& d ^2(\tilde i  _{2k} ( m) ) \\
 &= \Bigl[d_{\HH} \left( \sum_{i_1, i_2, \ldots , i_k=1}^{nd} 
 \sigma^{i_k } \left(\bar m\right)  \left(  x^{nd-i_k}  \left( \bar x ^{i_1+i_2+\cdots + i_k - k }\right)\right)
 \otimes \cdots \otimes\left(\bar x^{nd-i_{k-1}}\right)
\otimes \left(\bar x\right)  \right)\\
& \qquad- 
 d_{\HH}\left(d\Tr\left(\bar m\right)\left(x^{nd-n}\left(\bar x^{n-1}\right)\right)\otimes \left(\bar x\right)\otimes\cdots\otimes \left(\bar x^{nd-1}\right)\otimes\left(\bar x\right)\right)\Bigr].
 \end{align*}
The second $d_{\HH}$ has most summands canceling because $\bar x^{nd}=0$, and we are left with
\begin{align*}
& d_{\HH}\left(d\Tr\left(\bar m\right)\left(x^{nd-n}\left(\bar x^{n-1}\right)\right)\otimes \left(\bar x\right)\otimes\cdots\otimes \left(\bar x^{nd-1}\right)\otimes\left(\bar x\right)\right)\\
& = d\Tr\left(\bar m\right)\left(x^{nd-n}\left(\bar x^n\right)\right)\otimes \left(\bar x^{nd-1}\right)\otimes \cdots \otimes\left(\bar x^{nd-1}\right)\otimes \left(\bar x\right)\\
&\quad   -\bar x d\Tr\left(\bar m\right)\left(x^{nd-n}\left(\bar x^{n-1}\right)\right)\otimes \left(\bar x\right)\otimes\cdots\otimes \left(\bar x^{nd-1}\right).
\end{align*}
In $d_{\HH} \left( \sum_{i_1, i_2, \ldots , i_k=1}^{nd} 
 \sigma^{i_k } \left(\bar m\right)  \left(  x^{nd-i_k}  \left( \bar x ^{i_1+i_2+\cdots + i_k - k }\right)\right)
 \otimes \cdots \otimes\left(\bar x^{nd-i_{k-1}}\right)
\otimes \left(\bar x\right)  \right)$, the  cancellations $d_{0,\HH}- d_{1,\HH} =\cdots =d_{2k-4,\HH}- d_{2k-3,\HH} =0$ hold just as they did for  $d_{\HH}( \tilde i _{2k} \left( m\right))$, so we only need to study $d_{2k-2,\HH}- d_{2k-1,\HH}$:
\begin{align*}
 &d  _{\HH} \left( \sum_{i_1, i_2,\cdots, i_k=1}^{nd}
 \sigma^{i_k } (\bar m)  \blb  x^{n-i_k}  \blb \bar x ^{i_1+i_2+\cdots + i_k - k }\brb\brb
\!\otimes\!  \left(\bar x\right)  \!\otimes\!  \cdots \!\otimes\! \left(\bar x^{nd-i_{k-1}}\right)
\!\otimes\!  \left(\bar x\right)  \right)\\ 
& =\sum_{i_1, i_2,\cdots, i_k=1}^{nd}
 \sigma^{i_k } \left(\bar m\right)  \blb  x^{nd-i_k}  \blb \bar x ^{i_1+i_2+\cdots + i_k - k }\brb\brb
\!\otimes\!  \left(\bar x\right) \!\otimes\!  \cdots
\!\otimes\!  \left(\bar x\right) \!\otimes\!  \left(\bar x^{nd-i_{k-1} + 1}\right) \\
& - \sum_{i_1, i_2,\cdots, i_k=1}^{nd}
\bar x \sigma^{i_k } \left(\bar m\right)  \blb  x^{nd-i_k}  \blb \bar x ^{i_1+i_2+\cdots + i_k - k }\brb\brb
\!\otimes\!  \left(\bar x\right) \!\otimes\!  \cdots
\!\otimes\!  \left(\bar x\right) \!\otimes\!  \left(\bar x^{nd-i_{k-1}}\right)\\
& =\sum_{i_1, i_2,\cdots, i_k=1}^{nd}
 \sigma^{i_k } \left(\bar m\right)  \blb  x^{nd-i_k}  \blb \bar x ^{i_1+i_2+\cdots + i_k - k+1 }\brb\brb
\!\otimes\!  \left(\bar x\right) \!\otimes\!  \cdots
\!\otimes\!  \left(\bar x\right) \!\otimes\!  \left(\bar x^{nd-i_{k-1} }\right) \\
& - \sum_{i_1, i_2,\cdots, i_k=1}^{nd}
\sigma^{i_k -1} \left(\bar m\right)\bar x  \blb  x^{nd-i_k}  \blb \bar x ^{i_1+i_2+\cdots + i_k - k }\brb\brb
\!\otimes\!  \left(\bar x\right) \!\otimes\!  \cdots
\!\otimes\!  \left(\bar x\right) \!\otimes\!  \left(\bar x^{nd-i_{k-1}}\right)\\
& =\sum_{i_1, i_2,\cdots, i_k=1}^{nd}
 \sigma^{i_k -1} \left(\bar m\right)  \blb  x^{nd-i_k+1}  \blb \bar x ^{i_1+i_2+\cdots + i_k - k }\brb\brb
\!\otimes\!  \left(\bar x\right) \!\otimes\!  \cdots
\!\otimes\!  \left(\bar x\right) \!\otimes\!  \left(\bar x^{nd-i_{k-1} }\right) \\
& - \sum_{i_1,i_2, \cdots, i_{k-1}=1}^{nd}\bar m\left(x^{nd}\left(\bar x^{i_1+i_2+\cdots+i_{k-1}-\left(k-1\right)}\right)\right)\!\otimes\!  \left(\bar x\right) \!\otimes\!  \cdots
\!\otimes\!  \left(\bar x\right) \!\otimes\!  \left(\bar x^{nd-i_{k-1}}\right)\\
& - \sum_{i_1, i_2,\cdots, i_k=1}^{nd}
\sigma^{i_k -1} \left(\bar m\right)\bar x  \blb  x^{nd-i_k}  \blb \bar x ^{i_1+i_2+\cdots + i_k - k }\brb\brb
\!\otimes\!  \left(\bar x\right) \!\otimes\!  \cdots
\!\otimes\!  \left(\bar x\right) \!\otimes\!  \left(\bar x^{nd-i_{k-1}}\right),
\end{align*}
where the second equality involves shifting $i_{k-1}$ by $1$ in the first sum.  That does not change the value of the sum because the extra term and the term we lose are both zero because  $\bar x^{nd}=0$.  The third equality involves shifting $i_k$ by $1$ in the first sum, which requires the correction term below it for the extra term.
Observe that
\begin{align*}
& d_{\mathrm{bar}}\left(\sigma^{i_k-1}\left(\bar m\right)\left(x\left(x^{nd-i_k}\left(\bar x^{i_1+i_2+\cdots +i_k-k}\right)\right)\right)\right)\\
& \ = \sigma^{i_k-1}\left(\bar m\right)\bar x\left(x^{nd-i_k}\left(\bar x^{i_1+i_2+\cdots+i_k-k}\right)\right)- \sigma^{i_k-1}\left(\bar m\right)\left(x^{nd-i_k+1}\left(\bar x^{i_1+i_2+\cdots+i_k-k}\right)\right)\\
& \ + \sigma^{i_k-1}\left(\bar m\right)\left(x\left(\bar x^{nd+i_1+i_2+\cdots+i_{k-1}-k}\right)\right),\\
\end{align*}
so 
\begin{align*}
& \sum_{i_1,i_2, \cdots, i_{k}=1}^{nd} \sigma^{i_k-1}\left(\bar m\right)\left(x\left(\bar x^{nd+i_1+i_2+\cdots+i_{k-1}-k}\right)\right)\otimes \left(\bar x\right) \otimes \cdots
\otimes \left(\bar x\right) \otimes \left(\bar x^{nd-i_{k-1}}\right) \\
& - \sum_{i_1,i_2, \cdots, i_{k-1}=1}^{nd}\bar m\left(x^{nd}\left(\bar x^{i_1+i_2+\cdots+i_{k-1}-\left(k-1\right)}\right)\right)\otimes \left(\bar x\right) \otimes \cdots
\otimes \left(\bar x\right) \otimes \left(\bar x^{nd-i_{k-1}}\right)\\
& = d\Tr\left(\bar m\right)\left(x\left(\bar x^{nd-1}\right)\right)\otimes \left(\bar x\right) \otimes \cdots
\otimes \left(\bar x\right) \otimes \left(\bar x^{nd-1}\right)\\
& - \sum_{i_1,i_2, \cdots, i_{k-1}=1}^{nd}\bar m\left(x^{nd}\left(\bar x^{i_1+i_2+\cdots+i_{k-1}-\left(k-1\right)}\right)\right)\otimes \left(\bar x\right) \otimes \cdots
\otimes \left(\bar x\right) \otimes \left(\bar x^{nd-i_{k-1}}\right).\\
\end{align*}
Therefore
\begin{align*}
&d^2\left(\tilde i_{2k}\left(m\right)\right)  = [d\Tr\left(\bar m\right)\left(x\left(\bar x^{nd-1}\right)\right)\otimes \left(\bar x\right) \otimes \cdots
\otimes \left(\bar x\right) \otimes \left(\bar x^{nd-1}\right)\\
& - \sum_{i_1,i_2, \cdots, i_{k-1}=1}^{nd}\bar m\left(x^{nd}\left(\bar x^{i_1+i_2+\cdots+i_{k-1}-\left(k-1\right)}\right)\right)\otimes \left(\bar x\right) \otimes \cdots
\otimes \left(\bar x\right) \otimes \left(\bar x^{nd-i_{k-1}}\right)\\
& -  d\Tr\left(\bar m\right)\left(x^{nd-n}\left(\bar x^n\right)\right)\otimes \left(\bar x^{nd-1}\right)\otimes \cdots \otimes\left(\bar x^{nd-1}\right)\otimes \left(\bar x\right)\\
&  +\bar x d\Tr\left(\bar m\right)\left(x^{nd-n}\left(\bar x^{n-1}\right)\right)\otimes \left(\bar x\right)\otimes\cdots\otimes \left(\bar x^{nd-1}\right)].\\
\end{align*}
This can be simplified using 
\begin{align*}
& d_{\mathrm{bar}}\left(d\Tr\left(\bar m\right)\left(x\left(x^{nd-n}\left(\bar x^{n-1}\right)\right)\right)\right)\\
& = d\Tr\left(\bar m\right)\bar x\left(x^{nd-n}\left(\bar x^{n-1}\right)\right) - d\Tr\left(\bar m\right)\left(x^{nd-n+1}\left(\bar x^{n-1}\right)\right) + d\Tr\left(\bar m\right)\left(x\left(\bar x^{nd-1}\right)\right)\\
& = \bar xd\Tr\left(\bar m\right)\left(x^{nd-n}\left(\bar x^{n-1}\right)\right) - d\Tr\left(\bar m\right)\left(x^{nd-n+1}\left(\bar x^{n-1}\right)\right) + d\Tr\left(\bar m\right)\left(x\left(\bar x^{nd-1}\right)\right),\\
& d_{\mathrm{bar}}\left( d\Tr\left(\bar m\right)\left(x^{nd-n+1} \left( x^{n-1}\left(\bar 1\right)\right)\right) \right)\\
&\ = 
0-d\Tr\left( \bar m\right)\left( x^{nd}\left(\bar 1\right)\right) + d\Tr\left(\bar m\right)\left(x^{nd-n+1}\left(\bar x^{n-1}\right)\right),\\ 
& d_{\mathrm{bar}}\left( d\Tr\left(\bar m\right)\left(x^{nd-n} \left( x^{n}\left(\bar 1\right)\right)\right) \right)= 
0-d\Tr\left( \bar m\right)\left( x^{nd}\left(\bar 1\right)\right) + d\Tr\left(\bar m\right)\left(x^{nd-n}\left(\bar x^{n}\right)\right),\\ 
\end{align*}
the last two equations being true by the choice of $m$.
We get
\begin{align*}
& d ^2(\tilde i  _{2k} ( m) )
= [d\Tr\left(\bar m\right)\left(x^{nd-n+1}\left(\bar x^{n-1}\right)\right)\otimes \left(\bar x\right) \otimes \cdots
\otimes \left(\bar x\right) \otimes \left(\bar x^{nd-1}\right)\\
& - \sum_{i_1,i_2, \cdots, i_{k-1}=1}^{nd}\bar m\left(x^{nd}\left(\bar x^{i_1+i_2+\cdots+i_{k-1}-\left(k-1\right)}\right)\right)\otimes \left(\bar x\right) \otimes \cdots
\otimes \left(\bar x\right) \otimes \left(\bar x^{nd-i_{k-1}}\right)\\
& -  d\Tr\left(\bar m\right)\left(x^{nd-n}\left(\bar x^n\right)\right)\otimes \left(\bar x^{nd-1}\right)\otimes \cdots \otimes\left(\bar x^{nd-1}\right)\otimes \left(\bar x\right)]\\
& = [d\Tr\left(\bar m\right)\left(x^{nd}\left(\bar 1\right)\right)\otimes \left(\bar x\right) \otimes \cdots
\otimes \left(\bar x\right) \otimes \left(\bar x^{nd-1}\right)\\
& - \sum_{i_1,i_2, \cdots, i_{k-1}=1}^{nd}\bar m\left(x^{nd}\left(\bar x^{i_1+i_2+\cdots+i_{k-1}-\left(k-1\right)}\right)\right)\otimes \left(\bar x\right) \otimes \cdots
\otimes \left(\bar x\right) \otimes \left(\bar x^{nd-i_{k-1}}\right)\\
& -  d\Tr\left(\bar m\right)\left(x^{nd}\left(\bar 1\right)\right)\otimes \left(\bar x^{nd-1}\right)\otimes \cdots \otimes\left(\bar x^{nd-1}\right)\otimes \left(\bar x\right)].
\end{align*}
Lemma \ref{alsoneeded} allows us to identify the class of this element in the $E^2$-term as 
\begin{align*}
& d ^2(\tilde i  _{2k} ( m) )
= \tau \cdot  [d\Tr\left(\bar m\right)\otimes  \bar x  \otimes \cdots
\otimes  \bar x  \otimes  \bar x^{nd-1} \\
&\qquad\qquad - \sum_{i_1,i_2, \cdots, i_{k-1}=1}^{nd}\bar m
\bar x^{i_1+i_2+\cdots+i_{k-1}-\left(k-1 \right) }\otimes  \bar x  \otimes \cdots
\otimes  \bar x  \otimes  \bar x^{nd-i_{k-1}} \\
&\qquad\qquad -  d\Tr\left(\bar m\right)\otimes  \bar x^{nd-1} \otimes \cdots \otimes \bar x^{nd-1} \otimes  \bar x ] \in \tau \cdot \HH_{2k-2} (\Ap ),
\end{align*}
where we can identify the sum as being $-\overline{ i_{2k-2} (m) }$, the reduction modulo $p$ of the image of $m$ by Larsen's map.
Moreover, Larsen's comparison map $\pi_{2k-2}$ (defined below \cite[Equation (3.8.1)]{L95})
from the Hochschild complex to his small complex (which shows that the two are quasi-isomorphic) sends
$$(\bar a \otimes \bar x  \otimes \cdots
\otimes \bar x \otimes \bar x^{nd-1} - 
\bar a \otimes \bar x^{nd-1} \otimes \cdots \otimes \bar x^{nd-1} \otimes \bar x )\mapsto  \bar a-\bar a= 0$$
for all  $a\in A$.  So we get, just as we did in the odd case, that
$$d ^2(\tilde i  _{2k} (m) )=-\tau \cdot [ \overline{ i_{2k-2} (m) }] .$$
\end{proof}

\begin{lem}\label{alsoneeded}
In the iterated bar construction $B.( \Ap , A , \Ap )$ which calculates $\Tor_*^ A  ( \Ap ,  \Ap )$, for any $\bar a\in \Ap $ and $0\leq t <n$ the chain $\bar a \blb  x^n\blb \bar x ^t\brb\brb\in B_1( \Ap , A , \Ap )$ represents the homology class $\tau\cdot \bar a \bar x^t\in \Tor_1^ A  ( \Ap ,  \Ap )=\tau\cdot  \Ap $.
\end{lem}
\begin{proof}
The fact that  $\bar a \blb  x^n\blb \bar x ^t\brb\brb$ is a cycle follows directly from $\bar x ^n=0$.  To see what homology class it represents, as in the proof of \cite[Proposition 4.2]{LM00} we compare the free $ A $ resolutions of $ \Ap $
$$\xymatrix{
{\cdots}\ar[r]
 & { A  \blb  A  \blb  \Ap  \brb\brb }  \ar[r] ^{\quad d_0-d_1}\ar[d] ^{f_1}
 & { A  \blb  \Ap  \brb }\ar[r] \ar[d] ^{f_0}
& { \Ap } \ar[r] \ar[d]^ =
& 0\\
0\ar[r]
& A  \ar[r]^-{\cdot p}
& A  \ar[r]
& { \Ap } \ar[r]
& 0.
}
$$
We choose a (non-additive) section $s:\ M/(p) \to M$ of reduction modulo $p$ which satisfies $s(\bar 0)=0$, $s(\bar 1) = 1$.  Defining 
$f_0\Bigl( a \blb \sum_{i=0}^{n-1} \bar m_i \bar x^i\brb \Bigr)= a \sum_{i=0}^{n-1} s(\bar m^i) x^i,$
we get a map that makes the right square in the diagram commute.  Therefore, composing $f_0$ with $d_0-d_1$ will yield an element of $ A $ which is divisible by $p$ so we can define $f_1= {1\over p}\  f_0\circ (d_0 -d_1)$.   Now $f_1$ and $f_0$ (with zero maps $f_i=0$ for $i>1$) induce a chain homotopy equivalence between the resolutions, which will continue being a chain homotopy equivalence after tensoring over $ A $ with $ \Ap $.  Evaluating this,
$$f_1(a \blb  x^n\blb \bar x ^t\brb\brb) = {1\over p}\  f_0 (a x^n \blb \bar x ^t\brb - a \blb \bar x ^{n+t}\brb ) = 
 {1\over p}\  f_0 (a x^n \blb \bar x ^t\brb ) =  {1\over p} a x^n x^t = a x^t,$$
 since $x^n=p$.
Reducing mod $p$, we get $\tau \cdot \bar a \bar x^t  \in\Tor_1^ A  ( \Ap ,  \Ap )=\tau\cdot  \Ap $.
\end{proof}

\section{Identifying the Torsion} \label{lastsec}
\begin{thm}
\label{localmainthm}
Let $D$ be a finite-dimensional division algebra over $\Q_p$ and let $A$ be a maximal order in $D$.  Let $L$ be the center of $D$, and let $S$ be its valuation ring and $\F_S$ the residue field of $S$; we can write $S=R[\pi]/P(\pi)$, for $R$ unramified over $\Z_p$, $\pi$ a uniformizer of $S$, and $P$ an Eisenstein polynomial.   Assume that $D$ is of degree $n$ over $L$ (that is, of dimension $n^2$ over $L$).
Then
$$
\pi_*( \THH(A)^\wedge_p)\cong
\begin{cases}
S \oplus \F_S^{\oplus n-1} & *=0\\ 
S/(aP'(\pi)) & *=2a-1 > 0\\
\F_S^{\oplus n-1} &*=2a>0\\
 0 & *<0.\\
 \end{cases}
$$

\end{thm}
\begin{proof}
We will use the following spectral sequence:

\begin{lem}
If $S$ is a commutative $\Z_p$-algebra and $A$ is an $S$-algebra, we have a spectral 
sequence of $S$--modules
\begin{equation}
\label{goodseq}
E_{r,s}^2 = \HH_r^{S} (A;\pi_s (\THH(S; A) ^\wedge_p) ) \Rightarrow \pi_{r+s} (\THH(A)^\wedge_p ),
\end{equation}
which is multiplicative with respect to multiplication by $\pi_* (\THH(S) ^\wedge_p) $.  This multiplication
is the obvious one on the coefficients $\pi_* (\THH(S; A) ^\wedge_p) $ in the $E^2$-term, but it comes from the multiplication $\THH(S)\wedge \THH(A) \to \THH(A)$ that we have because
$S$ is the center of $A$.
\end{lem}

\begin{proof}
We take a model of $HS$ cofibrant over the sphere spectrum, and a model of $HA$ cofibrant over $HS$.  We start with
the spectral sequence of \cite[Corollary 3.3 ]{Li00},
\[
E_{r,s}^2 = \HH_r^{S} (A;\THH_s(S; A))\Rightarrow \THH_{r+s}(A).
\]
It is too large because $ \pi_*(\THH(S; A))$ contains large $\Q_p$-vector spaces in addition to 
$\pi_*( \THH(S; A)^\wedge_p)\cong \pi_*( \THH(S)^\wedge_p)\otimes_S A$, which is the part we are interested in.  The spectral sequence above comes from identifying 
$$\THH(A) 
\! \simeq \! HA \wedge_{HA\wedge HA^\op} HA 
\simeq HA  \wedge_{HA\wedge_{HS}  HA^\op}  ((HA\wedge_{HS}  HA^\op)
\wedge_{HA\wedge HA^\op} HA)$$
and observing that the  map induced by the obvious inclusions
$$ THH(S;A) = HS  \wedge_{HS\wedge HS^\op} HA 
\to (HA\wedge_{HS}  HA^\op)
\wedge_{HA\wedge HA^\op} HA$$
is a weak equivalence.  This gives a weak equivalence
$$\THH(A) \simeq HA  \wedge_{HA\wedge_{HS}  HA^\op}  THH(S;A).$$
If we $p$-complete $THH(S; A)$, since $ HA$ and $HA\wedge_{HS}  HA^\op$ are already $p$-complete and the maps between them commute with the $p$-completion, we will get that
$$\THH(A) ^\wedge_p\simeq HA  \wedge_{HA\wedge_{HS}  HA^\op}  \THH(S;A)^\wedge_p,$$
yielding the spectral sequence (\ref{goodseq})
that we need.  
 
 To see what the multiplication $\THH(S)\wedge \THH(A) \to \THH(A)$ does, we use the naturality of our construction for free $S$-algebras and compare the spectral sequence (\ref{goodseq}) with the analogous one we would get for $S$ in place of $A$: 
 $$E_{r,s}^2 = \HH_r^{S} (S;\pi_s (\THH(S; S) ^\wedge_p) ) \Rightarrow \pi_{r+s} (\THH(S)^\wedge_p ).
$$
Note that $E_{r,s}^2=0$ unless $r=0$, and we only have the copy of $\pi_*(\THH(S)^\wedge_p $ in the $0$'th column.
\end{proof}

We recall from  \cite[Theorem 5.1]{LM00} that 
$$
\pi_*( \THH(S)^\wedge_p)\cong
\begin{cases}
S & *=0\\ 
S/(aP'(\pi)) & *=2a-1 > 0\\
0 &*=2a>0\  \mathrm{ or}\ *<0.\\
 \end{cases}
$$
In odd dimensions, the result there  actually gives the inverse different ideal modulo $aS$, but that is isomorphic to $S$ modulo $a$ times the different ideal $(P'(\pi))$.  Since $A$ is free over $S$, this gives
$$\pi_*( \THH(S;A)^\wedge_p)\cong\pi_*(( \THH(S)\wedge_{HS} HA)^\wedge_p) \cong \pi_*( \THH(S)^\wedge_p)\otimes_S A .$$  Thus in the spectral sequence (\ref{goodseq}), we have
$$
E^2_{r, s}\cong
\begin{cases}
\HH^S_r(A) & s=0\\ 
\HH^S_r(A; A/(aP'(\pi))) & s=2a-1 > 0\\
0 &s=2a>0\  \mathrm{ or}\ s<0.\\
 \end{cases}
$$
In  Equation (\ref{HHlocal1}) above we used Larsen's result from \cite[Theorem 3.5]{L95} with the ground ring $R$ equal to the ring of integers in the maximal unramified extension of $\Q_p$ that is contained in  the center $L$ of $D$. But we can use as the ground ring the ring of integers in any extension of $\Q_p$ over which $L$ is purely ramified and in particular, we can use  the ring of integers in $L$ itself and take $R=S$.  In that case, \cite[Theorem 3.5]{L95} tells us that
\begin{equation}\label{HHoverS}
\HH^S_*(A)\cong
\begin{cases}
T/\pi \ker(\Tr_{T/S}) \cong S \oplus \F_S^{\oplus n-1} & *=0\\ 
0 & *=2a-1 > 0\\
\ker(\Tr_{T/S}) / \pi \cong \F_S^{\oplus n-1} &*=2a>0.\\
 \end{cases}
\end{equation}
Comparing with Equation (\ref{HHlocal1}), the difference here is that the uniformizer of $S$ remains $\pi$, but its minimal polynomial over $R=S$ is now $x-\pi$, and so has derivative equal to $1$, making the odd dimensional Hochschild homology vanish.  Using this and the Universal Coefficient Theorem over $S$, for any nontrivial ideal $(\pi^i)$ in $S$,
$$
\HH^S_*(A; A/(\pi^i))\cong
\begin{cases}
S/(\pi^i) \oplus \F_S^{\oplus n-1} & *=0\\ 
\F_S^{\oplus n-1}& *> 0.\\
 \end{cases}
$$

We now have to separate into two cases, according to whether the ideal $(aP'(\pi))$ is the trivial ideal or not for different $a$'s.  In the case where the center $S$ is ramified over $\Z_p$,  the ideal $(P'(\pi))$ is already nontrivial, and so for any integer $a\geq 1$, the ideal $(aP'(\pi))$ is nontrivial as well.  In that case,
the spectral sequence (\ref{goodseq}) takes the form 
\begin{equation}\label{goodseqramified}
E^2_{r, s}\cong
\begin{cases}
S \oplus \F_S^{\oplus n-1} & s=r=0\\ 
S/(aP'(\pi))  \oplus \F_S^{\oplus n-1} & s=2a-1  \  \mathrm{
 and}\ r=0\\ 
\F_S^{\oplus n-1}& s=2a-1  \  \mathrm{
 and}\ r>0, \  \!\mathrm{
 or}\ \!s=0 \  \mathrm{
 and}\ r=2b>0\\
0 &s=2a>0, \  \mathrm{
 or}\ s=0 \  \mathrm{
 and}\ r=2b-1>0.\\
 \end{cases}
\end{equation}
On the other hand, if the center $S$ is unramified over $\Z_p$, $(P'(\pi))=S$ and so the ideal $(aP'(\pi))$ will be nontrivial only when  $a$ is a multiple of $p$.  In that case, we get
\begin{equation}\label{goodsequnramified}
E^2_{r, s}\cong
\begin{cases}
S \oplus \F_S^{\oplus n-1} & s=r=0\\ 
S/(ap)  \oplus \F_S^{\oplus n-1} & s=2ap-1 \  \mathrm{ and}\ r=0\\ 
\F_S^{\oplus n-1}& s=2ap-1 \  \mathrm{ and}\ r>0, \  \!\mathrm{ or}\ \!s=0 \  \mathrm{ and}\ r=2b>0\\
0 &s=2a>0\  \mathrm{ or}\  s=2a-1  \  \mathrm{ and }\  p\nmid a, \\
 & \qquad  \mathrm{ or}\ s=0 \  \mathrm{ and}\ r=2b-1>0.\\
 \end{cases}
\end{equation}
 Since $S$ is the center of $A$, these are spectral sequences of $S$-modules and it is meaningful to talk of $S$-ranks.  In both cases, we have seen that $\pi_*( THH(A) ^\wedge _p)$ will consist entirely of torsion in positive dimensions, and of  torsion  and one copy of $S$ in dimension zero.   Any $S$-torsion of rank $m$ in $\pi_i( THH(A) ^\wedge _p)$ will give $S$-torsion of rank $m$ in dimensions $i$ and $i+1$ in $$\pi_*( THH(A)^ \wedge _p; \F_p) \cong \pi_*( THH(A); \F_p)  \cong \THH_*(A; A/p).$$
So we read backwards from Theorem \ref{onemodponenot} which describes the homotopy groups with mod $p$ coefficients. 
 In the ramified case, it says that $\THH_i(A; A/p)$ has rank $n$ over $S$ for any $i\geq 0$, meaning (because of the copy of $S$ in dimension zero) that $\THH_i(A)$ should have rank $n$ for $i=0$, rank $1$ for odd positive $i$, and rank $n-1$ for even positive $i$.  In the unramified case, by similar analysis  $\THH_i(A)$ should have rank $n$ for $i=0$, rank $1$ in dimensions is $2ap-1$ for $a>0$, rank $n-1$ for even positive $i$, and rank zero in odd dimensions that are not of the form $2ap-1$.

To understand exactly what the homotopy groups are, we will separate the discussion into cases according to whether $p$ divides $n$, the degree of $A$ over $S$, and according to whether $S$ is ramified.  The cases where $p$ does not divide the degree $n$ are easier, because  then in the calculation of  $\pi_*( THH(A) ^\wedge _p)$  using the spectral sequence (\ref{goodseq}) we can see in the $0$'th column of the spectral sequence a copy of $\pi_*( THH(S)^ \wedge _p)$ which exactly gives all the odd-dimensional homotopy groups in 
$\pi_*( THH(A) ^\wedge _p)$.  Therefore all other groups in odd total dimension have to be cancelled, but we will argue inductively that all outgoing spectral sequence differentials from them have to be trivial, and we will show by counting elements that the only way the superfluous odd total dimension groups can all be cancelled is if every incoming spectral sequence differential that can be nontrivial is as nontrivial as possible.  After that, all that remains in any positive even total dimension is $\F_S^{\oplus n-1}$, the smallest possible $S$-module of rank $n-1$, and so having a good understanding the $S$-torsion in odd dimensions completely determines the $S$-torsion in even dimensions.

When $p$ does divide $n$, we can still say that all of $\pi_{2a-1}( THH(A) ^\wedge _p)$ comes from the 
$0$'th column of the spectral sequence (\ref{goodseq}); as explained above, we also know that it has $S$-rank $1$.  This is shown to be enough to force the same cancellation pattern as in the case where $p$ does not divide $n$, and thus determines the even torsion as before.

\smallskip\noindent{\bf {The case $d>1$, $p\nmid n$:}}

Since $p\nmid n$, the $E^2$-term in the spectral sequence (\ref{goodseqramified}) contains a copy of 
$\pi_*( THH(S)^ \wedge _p)$ in the $0$'th column which is the isomorphic image of $\pi_*( THH(S)^ \wedge _p)$ under the map induced on the spectral sequence (\ref{goodseq}) by the inclusion $S\hookrightarrow A$.  This is because, as in the proof of Corollary \ref{tamedivisionalgeba}, we have an isomorphism of $S$-modules $T\cong S \oplus \ker(\Tr_{T/S})$ with the copy of $S$ equal to the image of the inclusion $S\hookrightarrow T$, that is: $\ker((1-\sigma^{-1})$.  Using this decomposition with the result in equation (\ref{HHoverS}), we see that it is in fact the image of $S$ under  the inclusion $S\hookrightarrow T $ that survives to be the $S$ in $\HH_0^S(A)\cong S\oplus \F_S^{\oplus n-1}$.  Larsen's comparison between the small complex and the Hochschild complex is the standard inclusion $T\hookrightarrow A$ in dimension zero, so  this $S$ in $\HH_0^S(A)$ is the image of $S=\HH_0^S(S)$ under  the inclusion $S\hookrightarrow A. $ Similarly, in $\HH_0^S(A; A/(\pi^i)) \cong S/(\pi^i)\oplus \F_S^{\oplus n-1}$ the $S/(\pi^i)$ is the image of the included $S$.

By Theorem \ref{onemodponenot} we know that if  $p\nmid n$, the image of 
$\THH_{2a-1}(S, S/(p)) \to \THH_{2a-1}(A, A/(p))$ is the part of $ \THH_{2a-1}(A, A/(p))$ which does not come from $\Tor(\THH_{2a-2}(A), \F_p)$: starting in dimension zero, we see that the $\F_S^{\oplus n-1}$'s in $\THH_{2a}(A, A/(p))$ come from $\THH_{2a}(A)\otimes \F_p$ but those in $\THH_{2a-1}(A, A/(p))$ come from $\Tor(\THH_{2a-2}(A), \F_p)$.  So we deduce that the generator over $S$ of $\THH_{2a-1}(A)$ comes from the image of $\THH_{2a-1}(S)$ under the inclusion $S\hookrightarrow A$.  That image is by the previous paragraph exactly the $\THH_{2a-1}(S)$ summand in $E^2_{0, 2a-1}$.  Once we have accounted for the generator over $S$ of $\THH_{2a-1}(A)$ in $E^2_{0, 2a-1}$,  we know that all other classes in  total dimension $2a-1$ have to die  by the time we get to the $E^\infty$ term, both the $\F_S^{\oplus n-1}$ in $E^2_{0, 2a-1}$ and the classes in $E^2_{i, 2a-1-i}$ 
for $i>0$.  A priori there could have been a nontrivial extension of $E^\infty_{0, 2a-1}$ by some $E^\infty_{i, 2a-1-i}$ 
for $i>0$ which is still of rank one over $S$, but if the generator of $\THH_{2a-1}(A)$ over $S$ is already in $E^\infty_{0, 2a-1}$, this cannot happen.

We work inductively: clearly $E^2_{0,0}\cong E^\infty_{0,0}\cong \THH_0(A)$.  In total dimension $1$, we know that the $\F_S^{\oplus n-1}$ in $E^2_{0,1}$ must die before the $E^\infty$ term, and the only way this could happen is for $d^2:E^2_{2,0} \to E^2_{0,1}$ to be injective and onto this summand.  That means that all that is left in total dimension $2$ is the $\F_S^{\oplus n-1}$ in $E^2_{1,1}$, which is the smallest possible $S$-module of rank $n-1$ so there cannot be any possible nontrivial incoming differentials from total dimension $3$ to make it yet smaller. 

More generally, at each stage $a$ with $ a>1$, it turns out that most of the elements in total dimension $2a-2$ were  used to eliminate elements of total dimension $2a-3$ and all that is left in total dimension $2a-2$  is a copy of  $\F_S^{\oplus n-1}$ in $E^2_{1,2a-3}$ which cannot be shrunk further while maintaining a rank of $n-1$ over $S$.  Therefore there can be no nontrivial spectral sequence differentials from total dimension $2a-1$ into total dimension $2a-2$, and all classes of total dimension $2a-1$ that need to be eliminated  need to be hit by spectral sequence differentials from total dimension $2a$.
In total dimension $2a-1$, we have $\THH_{2a-1}(S)$ in $E^2_{0, 2a-1}$ and $a$ copies of $\F_S^{\oplus n-1}$ that have to be eliminated, in 
$E^2_{0, 2a-1},\  E^2_{2, 2a-31},\  \cdots,\ E^2_{2a-2, 1}$.  In total dimension $2a$, we have $a+1$ copies of $\F_S^{\oplus n-1}$, in 
$E^2_{1, 2a-1},\  E^2_{3, 2a-31},\  \cdots,\ E^2_{2a-1, 1}$ and also in $E^2_{2a, 0}$.  Note however that there cannot be any nontrivial spectral sequence differentials $d^r$ out of $E^r_{1, 2a-1}$ for $r\geq 2$, so that copy has to last to $E^\infty$ and be the rank $n-1$ $S$-module we need in $\THH_{2a}(A)$.  But the remaining $a$ copies all need to be used to eliminate the $a$ copies in total dimension $2a-1$.

In order to utilize the $a$ copies in total dimension $2a$ to eliminate classes in total dimension $2a-1$, we must have  that $d^2: E^2_{2,0} \to E^2_{0, 1}$ is injective onto the $\F_S^{\oplus n-1}$ summand, that
$\xymatrix{d^2: E^2_{2b,0} \ar[r]^\cong   &E^2_{2b-2, 1}}$ for all $b>1$, and that $d^2$ vanishes everywhere else.  And afterwards, that $d^3$ is as nontrivial as it can be, meaning that $d^3: E^3_{3,2a-1} \to E^3_{0, 2a+1}$ is injective onto the $\F_S^{\oplus n-1}$ summand for all $a>1$, that
$\xymatrix{d^3: E^3_{2b+1, 2a-1} \ar[r]^\cong  &E^3_{2b-2, 2a+1}}$ for all $b>1$, $a>0$, and that $d^3$ vanishes everywhere else.  No nontrivial differentials from total dimension $2a$ into total dimension $2a-1$ are possible beyond that, so we get $S$-module isomorphisms
\begin{equation}\label{bestcase}
E^\infty_{r, s}\cong
\begin{cases}
S \oplus \F_S^{\oplus n-1} & r=s=0\\ 
S/(aP'(\pi)) & r=0\ \mathrm {and }\  s=2a-1 \\
\F_S^{\oplus n-1} &r=1\  \mathrm {and }\  s=2a-1 \\
 0 & \mathrm {otherwise.}\\
 \end{cases}
\end{equation}
Since every diagonal contains exactly one nontrivial $S$-module, there are no possible extensions and the result of Theorem \ref{localmainthm} follows in this case.

\smallskip\goodbreak\noindent{\bf {The case $d>1$, $p\mid n$:}}

In the previous case, we used the fact that $p\nmid n$ to show that $\THH_{2a-1}(A)$ consists only of the isomorphic image of $\THH_{2a-1}(S)$ in $E^2_{0, 2a-1} $ in the spectral sequence (\ref{goodseqramified}).  When $p\mid n$, that copy is no longer the isomorphic image of $\THH_{2a-1}(S)$ under the map induced by the inclusion $S\hookrightarrow A$, but we can still show that  $\THH_{2a-1}(A)$ contains only a copy of $\THH_{2a-1}(S)$ that sits inside $E^2_{0, 2a-1} $, and that was the essential feature that allowed us to deduce what all the differentials of the spectral sequence must do.

We now look at the map induced on the spectral sequence (\ref{goodseq})  by the inclusion $T\hookrightarrow A$, instead.  Recall that $T$ is unramified over $S$, so $\HH_*^S(T) \cong T$  when $*=0$ and vanishes everywhere else; similarly, 
$\HH_*^S(T; T/(\pi^i)) \cong T/(\pi^i)$  when $*=0$ and vanishes everywhere else.  So the spectral sequence (\ref{goodseq}) for $T$ is concentrated in the $0$'th column and gives the result we know, that $\THH_*(T)\cong T\otimes_S\THH_*(S)$.  

This means that the image of $\pi_*( THH(T)^ \wedge _p)$  in $\pi_*( THH(A)^ \wedge _p)$  under the map induced by the inclusion $T\hookrightarrow A$ is all in the $0$'th column of  the spectral sequence (\ref{goodseqramified}).  This explains  why all the nontrivial $E^2_{i, 2a-1-i}$'s for $i>0$ need to disappear before the $E^\infty$ term.
As far as $E^2_{0, 2a-1}$ goes, we know that $E^\infty_{0, 2a-1}$ is a submodule of $\pi_*( THH(A)^ \wedge _p)$ which has rank one over $S$.  Therefore  $E^\infty_{0, 2a-1}$ has rank one over $S$, as well.  But $E^2_{0, 2a-1}$ has rank $n$ over $S$, and of course all the outgoing spectral sequence differentials originating from it cannot hit anything so they are zero.  So an $S$-submodule of rank $n-1$ inside $E^2_{0, 2a-1}$ must be eliminated by incoming spectral sequence differentials.  The above argument shows, by counting elements, that the incoming differential that could cancel these (while making sure all the $E^2_{i, 2a-i}$'s for $i>0$ are cancelled) is $d^2: E^2_{2,0} \to E^2_{0, 1}$ if $a=1$ or $d^3: E^3_{3,2a-3} \to E^3_{0, 2a-1}$ if $a>1$.  The source of these differentials is, in both cases, 
$\F_S^{\oplus n-1}$; the target is $S/(aP'(\pi))  \oplus \F_S^{\oplus n-1}$.  If the cokernel needs to have rank one over $S$, the cokernel has to be isomorphic to $S/(aP'(\pi)) $.  If it happens that $S/(aP'(\pi)) \cong\F_S$, it might be that it is not the separate $S/(aP'(\pi)) $ that remains in the cokernel but one of the 
 $\F_S$'s in $\F_S^{\oplus n-1}$, but it is nevertheless a copy of $S/(aP'(\pi)) $ that remains.
 
 So the same cancellation pattern that we had in the $p\nmid n$ case must hold in the $p\mid n$ case, yielding the same $E^\infty$ result as Equation (\ref{bestcase}).
 
\smallskip\goodbreak\noindent{\bf {The case $d=1$:}}

If $S$ is unramified over $\Z_p$, or in the notation we are using: $S=R$, the spectral sequence (\ref{goodseq}) takes the form (\ref{goodsequnramified}), which is much sparser than (\ref{goodseqramified}).  However, the argument  that explains why $\pi_{2a-1}( THH(A)^ \wedge _p)$ should contain only a copy of $\pi_{2a-1}( THH(S)^ \wedge _p)$ in $E^2_{0, 2a-1}$ of (\ref{goodseq}) is the same.  As before, the argument depends on whether that
copy of $\pi_{2a-1}( THH(S)^ \wedge _p)$
is the isomorphic image of $\pi_{2a-1}( THH(S)^ \wedge _p)$ under the inclusion $S\hookrightarrow A$ as in the case $p\nmid n$ or not.  Note that when $d=1$,  $\pi_{2a-1}( THH(S)^ \wedge _p)\cong 0$ if $p\nmid a$ but is a rank one $S$-module if $p\mid a$, so the argument would sometimes be used about a rank one module and sometimes about the  zero module.  Also, the argument that shows by induction on $a$ that the only way to get nothing but a copy of $\pi_{2a-1}( THH(S)^ \wedge _p)$ in $E^2_{0, 2a-1}$ in total dimension $2a-1$ in the $E^\infty$ term  is to have maximally nontrivial spectral sequence differentials from total dimension $2a$ remains the same argument.  

The difference in the unramified case is what the pattern of the maximally nontrivial differentials is.
Here, $d^2, d^3, \ldots, d^{2p-1}$ all have to be trivial.  The first potentially nontrivial differential is $d^{2p}$, where we must have that $d^{2p}: E^{2p}_{2p,0} \to E^{2p}_{0, 2p-1}$ is injective onto a direct summand isomorphic to $\F_S^{\oplus n-1}$, that for all $b>p$
$\xymatrix{d^{2p}: E^{2p}_{2b,0} \ar[r]^\cong   &E^{2p}_{2b-2p, 2p-1}}$, and that $d^{2p}$ vanishes everywhere else.  Afterwards, we have that 
$d^{2p+1}: E^{2p+1}_{2p+1,2ap-1} \to E^{2p+1}_{0, 2(a+1)p-1}$ is injective onto a direct summand isomorphic to $\F_S^{\oplus n-1}$ for all $a>0$, that  
$\xymatrix{d^{2p+1}: E^{2p+1}_{2b+1, 2ap-1} \ar[r]^\cong  &E^{2p+1}_{2b-2p, 2(a+1)p-1}}$ for all $b>p$,  $a>0$, and that $d^{2p+1}$ vanishes everywhere else.  And then no further nontrivial differentials are possible and  $E^{2p+2}\cong E^\infty$. We get
\begin{equation}\label{notbadcase}
E^\infty_{r, s}\cong
\begin{cases}
S \oplus \F_S^{\oplus n-1} & r=s=0\\ 
S/(a) & r=0\ \mathrm {and }\  s=2a-1 \\
\F_S^{\oplus n-1} &r=2b,\ 0< b<p,\  \mathrm {and }\  s= 0\\
\F_S^{\oplus n-1} &r=2b+1,\ 0\leq b<p,\  \mathrm {and }\  s=2ap-1,\ a>0 \\
0 & \mathrm {otherwise,}\\
 \end{cases}
\end{equation}
and since every diagonal contains at most one nontrivial $S$-module, there are no possible extensions and the result of Theorem \ref{localmainthm} follows in this case as well.
 \end{proof}

\section{Combining the Local Calculations into the Final Result} \label{verylastsec}

We now combine the results of Theorem \ref{localmainthm} over all localizations of the center $V$ at nontrivial prime ideals $\P\subset V$.  As explained in the introduction and worked out in Corollary \ref{globalcor}, $\THH_0(U)\cong \HH_0(U)$ and we know what it is by the work of Larsen in \cite{L95}.  For $*>0$,  $\THH_*(U)$  is a $V$-module that consists entirely of torsion, so it is equal to the direct sum over all primes $p\in\Z$ of the $p$-torsion in the homotopy groups of
$\THH(U)^{\wedge}_p\simeq \THH(U^\wedge_{(p)})^{\wedge}_p$  \cite[Addendum 6.2]{HM97}.  Since 
$$U^\wedge_{(p)} \cong\bigoplus_{\P\subset V\mathrm{\ prime,}\ (p)\subseteq \P} U^\wedge_\P,$$
 this  
is in turn 
equal to the direct sum over all primes $\P\subset V$ with $(p) \subseteq \P$ of the $\P$ torsion in $\pi_*(\THH(U^\wedge_{\P})^\wedge_p)$, so we need to understand that for all such $p$ and $\P$.  

If $B$ is the simple algebra over $\mathbb{Q}$ that $U$ is a maximal order in and $C$ its center (which has $V$ as its ring of integers), the completion $B^\wedge_\P$ is a central simple $C^\wedge_\P$-algebra, and so $B^\wedge_\P$ is isomorphic to an $i_\P\times i_\P$ matrix ring on some central division algebra $D_\P$ over $C^\wedge_\P$, of degree $e_\P$.    But then by \cite[Theorem X.1]{Weil}, the valuation ring $U^\wedge_P$ must be  isomorphic to an $i_\P\times i_\P$ matrix ring over the valuation ring of $D_\P$.

If $e_\P=1$ then $U^\wedge_\P$ is   isomorphic to an $i_\P\times i_\P$ matrix ring over $V^\wedge_{\P}$ itself.  Then the Morita equivalence of \cite[Proposition 3.9]{BHM93} shows that  $\THH(U^\wedge_\P)$ and $\THH(V^\wedge_\P)$ are weakly equivalent.  The homotopy groups of the $p$-completion of the latter were calculated in \cite{LM00}, so we are done.  The positive dimensional homotopy groups 
$\pi_*(\THH(U^\wedge_\P)^\wedge_p))$ vanish if $*=2a>0$, and give the direct summand consisting of all the $\P$-torsion in $\THH_{2a-1}(V)$
 if $*=2a-1>0$.  Writing that the even dimensional homotopy groups are isomorphic to 
$ \F_\P ^{\oplus e_\P-1} $ is also correct, since $e_\P=1$.  When $e_\P=1$, $U$ is called unramified at $\P$, and in fact $U$ is unramified at all but finitely many prime ideals in $V$.

But if $U$ does ramify at $\P$, that is: if $e_\P>1$, we need Theorem \ref{localmainthm} for $D=D_\P$, $A$ a maximal order in $D$, 
$L=C^\wedge_\P$, $S=V^\wedge_{\P}$, and $n=e_\P$.  The ring $C^\wedge_\P$ is the center of $B^\wedge_\P$, but since the center of the ring of matrices over a ring  consists of multiples of the identity matrix by elements of the center of that ring, $C^\wedge_\P$ is also isomorphic to the center of $D_\P$.
 The result of Theorem \ref{localmainthm} in odd dimensions consists of $\pi_{2a-1}(\THH(V^\wedge_\P)^\wedge_p))$, which is the
$\P$-torsion in  $\THH_{2a-1} (V) $.  Gathering the results of Theorem \ref{localmainthm} for all prime ideals $\P\subset V$ where $U$ is ramified together with the observations above for prime ideals $\P\subset V$where $U$ is unramified,  we get

\begin{thm} \label{main}
Let $B$ be a simple algebra over $\mathbb{Q}$, and let $U$ be a maximal order in it.  Let $C$ be the center of $B$, and let $V$ be its ring of integers.   For every nontrivial prime ideal $\P\subset  V$, the completion $B^\wedge_\P$ is a central simple $C^\wedge_\P$-algebra, and so $B^\wedge_\P$ is isomorphic to a matrix ring on some central division algebra $D_\P$ over , of degree $e_\P$.  Let $\F_\P= V/\P$.  Then we have $V$-module isomorphisms
$$
\THH_*(U)\cong
\begin{cases}
V \oplus  \bigoplus_{\P\subset V \ \mathrm{prime}} \F_\P ^{\oplus e_\P-1} & *=0\\ 
 \THH_{2a-1} (V)  & *=2a-1> 0\\
  \bigoplus_{\P\subset V \ \mathrm{prime}} \F_\P ^{\oplus e_\P-1} & *=2a>0\\
 0 & *<0.
 \end{cases}
$$

\end{thm}

\begin{rem}\label{indep}
Note that the result here is entirely determined by the simple algebra $B$.  Knowing $B$ determines its center $C$ and $C$'s valuation ring  $V$.  Knowing $V$ determines its nontrivial prime ideals $\P\subset V$, and for each such $\P$, the  field $\F_\P=V/\P$ and the degree $e_\P$, which is the degree of the central $C^\wedge_\P$-division algebra $D_\P$    which $B^\wedge_\P$ is an $i_\P\times i_\P$ matrix ring over.
All of this is entirely independent of the choice of the maximal order $U$.  Thus while it is true that a simple algebra over $\mathbb{Q}$ might have different and non-isomorphic maximal orders, they must all have the same topological Hochschild homology.
 \end{rem}

\nocite{*}
\bibliographystyle{alpha}
\bibliography{THH}{}

\end{document}